\documentclass[11pt]{article}
\usepackage{a4wide}

\usepackage{amsmath,amsfonts,amssymb,amsthm,amscd}
\usepackage{epsfig}
\usepackage{hyperref}

\newtheorem{theorem}{Theorem}       
\newtheorem{proposition}[theorem]{Proposition}           
\newtheorem{lemma}[theorem]{Lemma}
\newtheorem{corollary}[theorem]{Corollary}

\newtheorem{observation}[theorem]{Observation}

\newtheorem{problem}{Problem}

\newtheorem*{claim_}{Claim}

\begin{document}

\title{Improved enumeration of simple topological graphs\thanks{The author was supported 
by the GraDR EUROGIGA GACR project No. GIG/11/E023 and
by the grant SVV-2013-267313 (Discrete Models and Algorithms). Part of the research was conducted during the Special Semester on Discrete and Computational Geometry at \'Ecole Polytechnique F\'ederale de Lausanne, organized and supported by the CIB (Centre  Interfacultaire Bernoulli) and the SNSF (Swiss National Science Foundation). Earlier version of the paper appeared in the author's doctoral thesis~\cite{K13_combinatorial}.
}
}

\author{
Jan Kyn\v cl
} 

\date{}

\maketitle

\begin{center}
{\small \footnotesize
Department of Applied Mathematics and Institute for Theoretical Computer Science, \\
Charles University, Faculty of Mathematics and Physics, \\
Malostransk\'e n\'am.~25, 118~ 00 Praha 1, Czech Republic; \\
\texttt{kyncl@kam.mff.cuni.cz}
}
\end{center}

\begin{abstract}

A {\em simple topological graph\/} $T=(V(T), E(T))$ is a drawing of a graph in the plane where every two edges have at most one common point (an endpoint or a crossing) and no three edges pass through a single crossing. Topological graphs $G$ and $H$ are {\em isomorphic\/} if $H$ can be obtained from $G$ by a homeomorphism of the sphere, and {\em weakly isomorphic\/} if $G$ and $H$ have the same set of pairs of crossing edges.

We generalize results of Pach and T\'oth and the author's previous results on counting different drawings of a graph under both notions of isomorphism.
We prove that for every graph $G$ with $n$ vertices, $m$ edges and no isolated vertices the number of weak isomorphism classes of simple topological graphs that realize $G$ is at most $2^{O(n^2\log (m/n))}$, and at most $2^{O(mn^{1/2}\log n)}$ if $m\le n^{3/2}$. As a consequence we obtain a new upper bound $2^{O(n^{3/2}\log n)}$ on the number of intersection graphs of $n$ pseudosegments. We improve the upper bound on the number of weak isomorphism classes of simple complete topological graphs with $n$ vertices to $2^{n^2\cdot \alpha(n)^{O(1)}}$, using an upper bound on the size of a set of permutations with bounded VC-dimension recently proved by Cibulka and the author. We show that the number of isomorphism classes of simple topological graphs that realize $G$ is at most $2^{m^2+O(mn)}$ and at least $2^{\Omega(m^2)}$ for graphs with $m>(6+\varepsilon)n$. 
\end{abstract}


\section{Introduction}\label{section_intro}

A {\em topological graph\/} $T=(V(T), E(T))$ is a drawing of a graph $G$ in the plane with the following properties. The vertices of $G$ are represented by a set $V(T)$ of distinct points in the plane and the edges of $G$ are represented by a set $E(T)$ of simple curves connecting the corresponding pairs of points. We call the elements of $V(T)$ and $E(T)$ the {\em vertices\/} and the {\em edges\/} of $T$. The drawing has to satisfy the following general position conditions: (1) the edges pass through no vertices except their endpoints, (2) every two edges have only a finite number of intersection points, (3) every intersection point of two edges is either a common endpoint or a proper crossing (``touching'' of the edges is not allowed), and (4) no three edges pass through the same crossing. A topological graph is {\em simple\/} if every two edges have at most one common point, which is either a common endpoint or a crossing. A topological graph is {\em complete\/} if it is a drawing of a complete graph. 

We use two different notions of isomorphism to enumerate topological graphs.

Topological graphs $G$ and $H$ are {\em weakly isomorphic\/} if there exists an incidence preserving one-to-one correspondence between $V(G), E(G)$ and $V(H), E(H)$ such that two edges of $G$ cross if and only if the corresponding two edges of $H$ do. 

Note that every topological graph $G$ drawn in the plane induces a drawing $G_{S^2}$ on the sphere, which is obtained by a standard one-point compactification of the plane.  
Topological graphs $G$ and $H$ are {\em isomorphic\/} if there exists a homeomorphism of the sphere which transforms $G_{S^2}$ into $H_{S^2}$. 
In Section~\ref{section_horni_strong} we give an equivalent combinatorial definition.

Unlike isomorphism, weak isomorphism can change the faces of the topological graphs involved, the order of crossings along the edges and also the cyclic orders of edges around vertices. 

For counting (weak) isomorphism classes, we consider all graphs labeled. That is, each vertex is assigned a unique label from the set $\{1,2,\dots,n\}$, and we require the (weak) isomorphism to preserve the labels. Mostly it makes no significant difference in the results as we operate with quantities asymptotically larger than $n!$.


For a graph $G$, let $T_\mathrm{w}(G)$ be the number of weak isomorphism classes of simple topological graphs that realize $G$. 
Pach and T\'oth~\cite{PT04_how} and the author~\cite{K06_crossings} proved the following lower and upper bounds on $T_\mathrm{w}(K_n)$.

\begin{theorem}\label{veta_uplne_historicka} {\rm \cite{K06_crossings, PT04_how}} For the number of weak isomorphism classes of simple drawings of $K_n$, we have
$$2^{\Omega(n^2)} \le T_\mathrm{w}(K_n) \le ((n-2)!)^n = 2^{O(n^2\log{n})}.$$
\end{theorem}

We prove generalized upper and lower bounds on $T_\mathrm{w}(G)$ for all graphs $G$. 

\begin{theorem}\label{veta_hlavni}
Let $G$ be a graph with $n$ vertices and $m$ edges. Then
$$T_\mathrm{w}(G) \le 2^{O(n^{2}\log (m/n))}.$$
If $m<n^{3/2}$, then
$$T_\mathrm{w}(G) \le 2^{O(mn^{1/2}\log n)}.$$
Let $\varepsilon>0$. If $G$ is a graph with no isolated vertices and at least one of the conditions $m>(1+\varepsilon)n$ or $\Delta(G)<(1-\varepsilon)n$ is satisfied, then
$$T_\mathrm{w}(G) \ge 2^{\Omega(\max (m, n\log{n}))}.$$
\end{theorem}

We also improve the upper bound from Theorem~\ref{veta_uplne_historicka}.

\begin{theorem}\label{veta_uplne} We have 
$$T_\mathrm{w}(K_n) \le 2^{n^2\cdot \alpha(n)^{O(1)}}.$$
\end{theorem}

Here $\alpha(n)$ is the inverse of the Ackermann function. It is an extremely slowly growing function, which can be defined in the following way~\cite{N10_improved}.
$\alpha(m) := \min\{k: \alpha_k(m) \leq 3\}$ where $\alpha_d(m)$ is the $d$th function in the \emph{inverse Ackermann hierarchy}. That is, $\alpha_1(m)=\lceil m/2 \rceil$,
$\alpha_d(1)=0$ for $d \geq 2$ and $\alpha_d(m) = 1+ \alpha_d(\alpha_{d-1}(m))$ for $m,d  \geq 2$. The constant in the $O(1)$ notation in the exponent is huge (roughly $4^{30^4}$), due to a Ramsey-type argument used in the proof.

Theorem~\ref{veta_uplne} is proved in Section~\ref{section_uplne}. In the proof of Theorem~\ref{veta_uplne} we use the fact that for simple complete topological graphs, the weak isomorphism class is determined by the rotation system~\cite{K09_enumeration, PT04_how} (see Proposition~\ref{tvrz_1graf-2rot}). This is combined with a recent combinatorial result, an upper bound on the size of a set of permutations with bounded VC-dimension~\cite{CK12_zobecnena_SW} (Theorem~\ref{veta_vcdimenze}). The method in the proof of Theorem~\ref{veta_hlavni} is more topological, gives a slightly weaker upper bound, but can be generalized to all graphs.

In Subsection~\ref{section_combinatorial}, we generalize Theorem~\ref{veta_uplne} by removing almost all topological aspects of the proof. The resulting Theorem~\ref{veta_kombinatoricka} is a purely combinatorial statement.

In Subsection~\ref{sub_max_crossings}, we consider the class of simple complete topological graphs with maximum number of crossings and suggest an alternative method for obtaining an upper bound on the number of weak isomorphism classes of such graphs.

An arrangement of {\em pseudosegments\/} (or also {\em $1$-strings\/}) is a set of simple curves in the plane such that any two of the curves cross at most once. An {\em intersection graph of pseudosegments\/} (also called a {\em string graph of rank $1$\/}) is a graph $G$ such that there exists an arrangement of pseudosegments with one pseudosegment for each vertex of $G$ and a pair of pseudosegments crossing if and only if the corresponding pair of vertices forms an edge in $G$. Using tools from extremal graph theory, Pach and T\'oth~\cite{PT04_how} proved that the number of intersection graphs of $n$ pseudosegments is $2^{o(n^2)}$. 
As a special case of Theorem~\ref{veta_hlavni} we obtain the following upper bound. 

\begin{theorem}\label{veta_pseudosegmenty}
There are at most $2^{O(n^{3/2}\log n)}$ intersection graphs of $n$ pseudosegments.
\end{theorem}

The best known lower bound for the number of (unlabeled) intersection graphs of $n$ pseudosegments is $2^{\Omega(n\log n)}$. This follows by a simple construction or from the the fact that there are $2^{\Theta(n\log n)}$ nonisomorphic permutation graphs with $n$ vertices~\cite{BG09_localized}. 

Let $T(G)$ be the number of isomorphism classes of simple topological graphs that realize $G$. 
A sequence $G_1,G_2,\dots$ of graphs where $G_n$ has $n$ vertices and $m=m(n)$ edges has {\em superlinear\/} number of edges if $m(n)>\omega(n)$, that is, if for every constant $c$, we have $m(n)<cn$ for sufficiently large $n$.
The following theorem generalizes the result $T(K_n)= 2^{\Theta(n^4)}$ from~\cite{K09_enumeration}.

\begin{theorem}\label{veta_neizo}
Let $G$ be a graph with $n$ vertices, $m$ edges and no isolated vertices. Then $T(G) \le 2^{m^2+O(mn)}$. More precisely,
\begin{align*}
 1) \text{ } T(G)  &\le 2^{m^2+11.51mn+O(n\log n)} \text{ and } T(G) \le 2^{23.118 m^2}+o(1), \\
 2) \text{ } T(G)  &\le 2^{m^2+2mn(\log(1+\frac{m}{4n})+3.443) +O(n\log n)} \text{ and } T(G) \le 2^{11.265 m^2}+o(1).  
\end{align*}
Let $\varepsilon>0$. For graphs $G$ with $m>(6+\varepsilon)n$ we have
$$T(G) \ge 2^{\Omega(m^2)}.$$
For a sequence of graphs $G_n$ with superlinear number of edges we have
$$T(G_n) \ge 2^{m^2/60}-o(1).$$
\end{theorem}

The two upper bounds on $T(G)$ come from two essentially different approaches. The first one gives better asymptotic results for dense graphs, whereas the second one is better for sparse graphs (roughly, with at most $17n$ edges). 
For such very sparse graphs (for example, matchings), however, better upper bounds can be deduced more directly from other known results; see the discussion in Subsection~\ref{sub_very_sparse}.

The proof in~\cite{K09_enumeration} implies the upper bound $T(K_n)\le 2^{(1/12+o(1))(n^4)}$, although it is not explicitly stated in the paper. However, the key Proposition $7$ in~\cite{K09_enumeration} is incorrect. We prove a correct version in Section~\ref{section_horni_strong}.

In Section~\ref{section_geometric} we briefly discuss the special case of geometric graphs.

All the logarithms used in this paper are binary, unless indicated otherwise.


\section{Preliminaries}

The weak isomorphism classes of topological graphs can be represented in a combinatorial way by abstract topological graphs. An {\em abstract topological graph\/} (or briefly an {\em AT-graph}) is a pair $(G,R)$ where $G$ is a graph and $R \subseteq {E(G) \choose 2}$ is a set of pairs of its edges. For a topological graph $T$ that is a drawing of $G$ we define the AT-graph of $T$ as $(G,R_T)$ where $R_T$ is the set of pairs of edges having at least one common crossing. A (simple) topological graph $T$ is called a {\em (simple) realization\/} of $(G,R)$ if $R_T=R$. 
Clearly, two topological graphs are weakly isomorphic if and only if they are realizations of the same AT-graph. 

The {\em rotation\/} of a vertex $v$ in a topological graph $T$ is the clockwise cyclic order of the edges incident with $v$. The rotation $\rho(v)$ of a vertex $v$ is represented by a cyclic sequence of the vertices adjacent to $v$. The {\em rotation system\/} of $T$ is the set of rotations of all its vertices. 

We use the following property of simple complete topological graphs, which directly implies the upper bound on $T_\mathrm{w}(K_n)$ in Theorem~\ref{veta_uplne_historicka}. 

\begin{proposition}\label{tvrz_1graf-2rot}{\rm \cite{K09_enumeration, PT04_how}}
The rotation system of a simple complete topological graph $G$ uniquely determines which pairs of edges of $G$ cross. That is, two simple complete topological graphs with the same rotation system are weakly isomorphic.
\end{proposition}

This property is not satisfied for arbitrary graphs. For complete bipartite graphs, many weakly nonisomorphic drawings can share the same rotation system. For example, there are at least $2^{n/2}$ weakly nonisomorphic simple drawings of $K_{2,n}$ with the same rotation system. To see this, let $n$ be an even positive integer and let $v,w,u_1,u_2, \dots, u_n$ be the vertices of $K_{2,n}$ with $v,w$ forming the $2$-element independent set of the bipartition. Let $(u_1,u_2,\dots,u_n)$ be the rotation of $v$ and $(u_{n-1},u_n,\dots,u_3,u_4,u_1,u_2)$ the rotation of $w$. For every $i=1,2,\dots,n/2$, there are two ways of drawing the four edges $vu_{2i},vu_{2i-1},wu_{2i},wu_{2i-1}$ (either $vu_{2i-1}$ crosses $wu_{2i}$ or $wu_{2i-1}$ crosses $vu_{2i}$), and these choices can be done independently. See Figure~\ref{obr_0_C4_dvema_zpusoby}. Note that by cloning the vertex $v$ into $n-1$ copies we obtain $2^{n/2}$ weakly nonisomorphic drawings of $K_{n,n}$ with the same rotation system.

\begin{figure}
\begin{center}
\epsfig{file={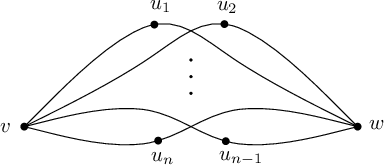}, scale=1}
\end{center}
\caption{Every $4$-cycle $vu_{2i-1}wu_{2i}$ in $K_{2,n}$ can be drawn in one of two ways, while keeping the rotation system fixed.}
\label{obr_0_C4_dvema_zpusoby}
\end{figure}

We note that the converse of Proposition~\ref{tvrz_1graf-2rot} is also true: the rotation systems of two weakly isomorphic simple complete topological graphs are either the same or inverse~\cite{G05_complete,K09_enumeration}.


\section{Simple complete topological graphs}\label{section_uplne}

In this section we prove Theorem~\ref{veta_uplne}.

The upper bound $T_\mathrm{w}(K_n) \le 2^{O(n^2\log{n})}$ in Theorem~\ref{veta_uplne_historicka} follows directly from Proposition~\ref{tvrz_1graf-2rot}, since there are at most $(n-2)!$ possible rotations for each vertex, thus at most $((n-2)!)^n = 2^{O(n^2\log{n})}$ possible rotation systems of $K_n$. However, not every set of rotations is realizable as a rotation system of a simple complete topological graph. For example, the rotation of each vertex in a simple complete topological graph is uniquely determined by the set of rotations of the other $n-1$ vertices. This is easily seen by investigating the drawings of $K_4$~\cite{K06_crossings, PT04_how} (see Observation~\ref{obs_8_rotacnich_systemu}) and using the fact that a cyclic permutation of $n$ elements is determined by cyclic subpermutations of all triples. 

The smallest forbidden patterns in the rotation system are the $4$-tuples of cyclic subpermutations of $3$ elements that cannot be realized as rotation systems of a simple drawing of $K_4$. In fact, in Section~\ref{section_combinatorial} we show that it is possible to prove Theorem~\ref{veta_uplne} by combinatorial arguments, using only these simple forbidden patterns.

However, we first show a proof relying more on the topological structure of the drawings, which gives a better upper bound on $T_\mathrm{w}(K_n)$, and also provides an intuition for the purely combinatorial proof.

The core idea in both versions of the proof is to reduce the problem of bounding $T_\mathrm{w}(K_n)$ to counting single permutations with forbidden subpermutations.

\subsection{Permutations with bounded VC-dimension}

Let $S_n$ be the set of all {\em $n$-permutations}, that is, permutations of the set $\{1,2,\dots,$ $n\}$.
The {\em restriction of $\pi \in S_n$ to the $k$-tuple $(a_1, a_2, \dots, a_k)$ of positions\/}, 
where $1\leq a_1<a_2< \dots <a_k\leq n$, is the $k$-permutation 
$\pi'$ satisfying $\forall i,j: \pi'(i) < \pi'(j) \Leftrightarrow \pi(a_i) < \pi(a_j)$. Let $\mathcal P \subseteq S_n$.
The $k$-tuple of positions $(a_1, \dots, a_k)$ is {\em shattered by $\mathcal P$} if each
$k$-permutation appears as a restriction of some $\pi \in \mathcal P$ to $(a_1, \dots, a_k)$.
The {\em VC-dimension of $\mathcal P$} is the size of the largest set of positions shattered by $\mathcal P$. In other words, the VC-dimension of $\mathcal P$ is at most $k$ if for every $k+1$ positions $a_1, \dots, a_{k+1}$ there is some forbidden $(k+1)$-permutation that does not appear as a restriction of any $\pi\in \mathcal P$ to $(a_1, \dots, a_{k+1})$. Raz~\cite{R00_VC} proved that a set of $n$-permutations of VC-dimension $2$ has size at most $2^{O(n)}$.
The following result proved by Cibulka and the author~\cite{CK12_zobecnena_SW} is the key ingredient in the proof of Theorem~\ref{veta_uplne}.

\begin{theorem}\label{veta_vcdimenze}{\rm \cite{CK12_zobecnena_SW}}
For every $t\ge 2$, the size of a set of $n$-permutations with VC-dimension $2t+2$ is at most
$$2^{n \cdot ((2/t!)\alpha(n)^t + O(\alpha(n)^{t-1}) )}.$$
\end{theorem}


The upper bound in Theorem~\ref{veta_vcdimenze} is asymptotically almost tight, since there are sets of permutations with VC-dimension $2t+2$ of size $2^{n \cdot ((1/t!)\alpha(n)^t - O(\alpha(n)^{t-1}) }$~\cite{CK12_zobecnena_SW}.

If the forbidden $(k+1)$-permutation is the same for all $(k+1)$-tuples of positions, we get a better, exponential upper bound on the size of $\mathcal P$. This was conjectured by Stanley and Wilf and proved by Marcus and Tardos~\cite{MT04_SW}, using Klazar's earlier result~\cite{K00_FH_SW}.
Later Cibulka~\cite{C09_on_constants} improved Klazar's reduction and obtained the upper bound $2^{O(k \log k)n}$ on the size of $\mathcal P$.

\subsection{Unavoidable topological subgraphs}

A {\em complete convex geometric graph\/} (shortly a {\em convex graph}) is a topological graph whose vertices are in convex position and the edges are drawn as straight-line segments; see Figure~\ref{obr_1}, left. We denote by $C_m$ any complete convex geometric graph with $m$ vertices, as all such graphs belong to the same weak isomorphism class. 

A simple complete topological graph with $m$ vertices is called {\em twisted} and denoted by $T_m$ if there exists a {\em canonical} ordering of its vertices $v_1, v_2, \dots, v_m$ such that for every $i<j$ and $k<l$ two edges $v_iv_j, v_kv_l$ cross if and only if $i<k<l<j$ or $k<i<j<l$; see Figure~\ref{obr_1}, right. Figure~\ref{obr_2} shows an equivalent drawing of $T_m$ on the cylindrical surface.

Let $G$ and $H$ be topological graphs. We say that $G$ {\em contains} $H$ if $G$ has a topological subgraph weakly isomorphic to $H$.

\begin{figure}
\begin{center}
\epsfig{file=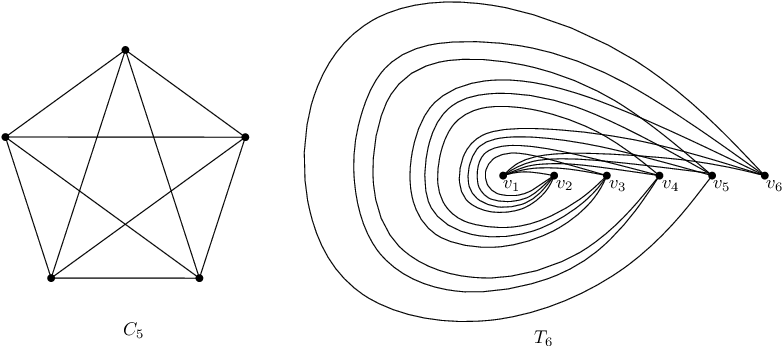, scale=1}
\end{center}
\caption{The convex graph $C_5$ and the twisted graph $T_6$.}
\label{obr_1}
\end{figure}

\begin{figure}
\begin{center}
\epsfbox{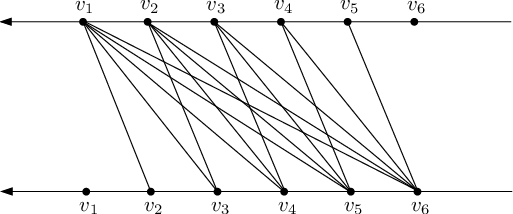}
\end{center}
\caption{A drawing of the twisted graph $T_6$ on the cylindrical surface.}
\label{obr_2}
\end{figure}

We use the following asymmetric form of the Ramsey-type result by Pach, Solymosi and T\'oth~\cite{PST03_unavoidable}, which generalizes the Erd\H{o}s--Szekeres theorem for planar point sets.

\begin{theorem}
\label{theorem_Cm1_or_Tm2}{\rm \cite{PST03_unavoidable}}
For all positive integers $n, m_1, m_2$ satisfying 
$$m_1m_2\le \log_4^{1/4}(n+1),$$ 
every simple complete topological graph with $n$ vertices contains $C_{m_1}$ or $T_{m_2}$. 
\end{theorem}

The graphs $C_m$ and $T_m$ are special cases of simple complete topological graphs with $m$ vertices and ${m \choose 4}$ crossings, which is the maximum number of crossings possible~\cite{HM92_drawings}. The existence of a complete subgraph with $m$ vertices and ${m \choose 4}$ crossings in a sufficiently large simple complete topological graph $G$ follows directly from Ramsey's theorem and the nonplanarity of $K_5$~\cite{HMS95_ramsey_Kn}, but the bound on the size of $G$ obtained is much larger than that from Theorem~\ref{theorem_Cm1_or_Tm2}. For the special case $m_1=m_2=5$, Harborth, Mengersen, and Schelp~\cite{HMS95_ramsey_Kn} showed a much better upper bound than that following from Theorem~\ref{theorem_Cm1_or_Tm2}.

\begin{theorem}
\label{theorem_C5_or_T5}{\rm \cite{HMS95_ramsey_Kn}}
Every simple complete topological graph with $113$ vertices contains $C_5$ or $T_5$.
\end{theorem}

\subsection{Forbidden patterns in the rotation system}

Let $G$ be a simple complete topological graph and let $v$ be a vertex of $G$. Our goal is to obtain an upper bound on the number of possible rotations of $v$ in $G$ when the complete subgraph $G-v$ is fixed.
To this end, we need to identify some forbidden permutations in the rotation of $v$.

\begin{lemma}\label{lemma_triangle}
Let $G$ be a simple complete topological graph with vertices $1,2,3,4$. Suppose that the counter-clockwise order of the vertices of the topological triangle $123$ is $1,2,3$. If 
\begin{enumerate}
\item[(a)] the vertex $4$ is outside the triangle $123$ and its rotation is $(1,2,3)$, or
\item[(b)] the vertex $4$ is inside the triangle $123$ and its rotation is $(1,3,2)$, 
\end{enumerate}
then $G$ has no crossings. Otherwise $G$ has one crossing.
\end{lemma}

\begin{proof}
Figure~\ref{obr_3} shows representatives of all four isomorphism classes of simple complete topological graphs with vertices $1,2,3,4$. The notions of isomorphism and weak isomorphism for these graphs coincide, since in each of the four drawings different pairs of edges cross. Each of the drawings is chosen so that the vertices $1,2,3$ appear in counter-clockwise order in the triangle $123$ and the vertex $4$ is outside the triangle $123$. This still leaves some freedom in choosing the outer face of the drawing: we may always choose any of the three faces adjacent to the vertex $4$, but the rotation system of the drawing stays the same. Since the rotation of the vertex $4$ is $(1,2,3)$ in $H_1$, which is without crossings, and $(1,3,2)$ in $H_2,H_3$ and $H_4$, which have one crossing, the case when the vertex $4$ is in the outer face of $123$ follows. The other case follows by the symmetry exchanging the outer and the inner face of the triangle $123$.
\end{proof}

\begin{figure}
\begin{center}
\epsfbox{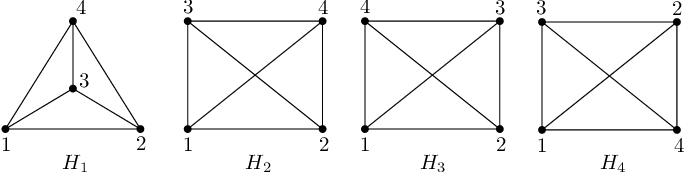}
\end{center}
\caption{Four nonisomorphic simple drawings of $K_4$.}
\label{obr_3}
\end{figure}

\begin{lemma}\label{lemma_C5}
Let $G$ be a simple complete topological graph with vertices $1,2,\dots,6$. Suppose that $G$ contains a convex graph $C_5$ induced by the vertices $1,2,\dots,5$, which appear in this counter-clockwise order on its outer face. Then the rotation of the vertex $6$ is not $(1,4,2,5,3)$. 
\end{lemma}

\begin{proof}
Let $H$ be the induced convex graph $G[\{1,2,3,4,5\}]$. Suppose for contradiction that the rotation of the vertex $6$ in $G$ is $(1,4,2,5,3)$. We distinguish two cases according to the face of $H$ in which the vertex $6$ is contained. See Figure~\ref{obr_4}.
\begin{enumerate}

\item [a)] The vertex $6$ is in one of the inner faces of $H$. By symmetry, we may assume that it is either in the inner pentagonal face or in the intersection of the triangles $234$ and $134$. The rotation of the vertex $6$ in $G[\{1,3,4,6\}]$ is $(1,4,3)$. By Lemma~\ref{lemma_triangle} applied to the triangle $134$, the edge $61$ lies completely inside the triangle $134$.
The vertex $6$ is also outside the triangle $125$ and the rotation of $6$ in $G[\{1,2,5,6\}]$ is $(1,2,5)$. By Lemma~\ref{lemma_triangle}, the edges $61$ and $25$ do not cross. But this is a contradiction as the vertices $6$ and $1$ are separated by a closed curve formed by portions of the edges $25,14,43,31$, which the edge $16$ cannot cross.

\item [b)] The vertex $6$ is in the outer face of $H$. By Lemma~\ref{lemma_triangle} applied to the triangle $125$, the edge $61$ cannot cross the edge $25$. Consequently, the edge $61$ crosses no edge of $H$. Similarly, no other edge adjacent to $6$ can cross an edge of $H$. This contradicts the conclusion of Lemma~\ref{lemma_triangle} applied to the triangle $134$. \qedhere
\end{enumerate}
\end{proof}

\begin{figure}
\begin{center}
\epsfbox{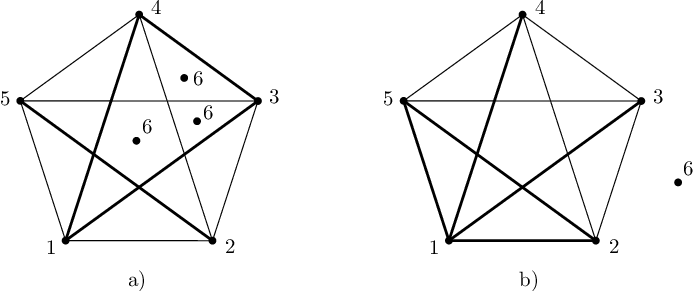} 
\end{center}
\caption{Impossibility of adding a vertex with rotation $(1,4,2,5,3)$. The thick edges cannot be crossed by the edge $61$.}
\label{obr_4}
\end{figure}

\begin{lemma}\label{lemma_C4}
Let $G$ be a simple complete topological graph with vertices $1,2,\dots,5$. Suppose that $G$ contains a convex graph $H$ induced by the vertices $1,2,3,4$, which appear in this counter-clockwise order on its outer face. If the vertex $5$ is inside the triangular face of $H$ adjacent to the vertices $2$ and $3$, then its rotation is not $(1,3,2,4)$.
\end{lemma}

\begin{figure}
\begin{center}
\epsfbox{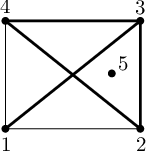} 
\end{center}
\caption{Impossibility of adding a vertex with rotation $(1,3,2,4)$ to the triangular face adjacent to the vertices $2$ and $3$. The thick edges cannot be crossed by the edge $54$.}
\label{obr_5}
\end{figure}

\begin{proof}
See Figure~\ref{obr_5}. Suppose for contradiction that the vertex $5$ is inside the triangular face of $H$ adjacent to the vertices $2$ and $3$ and its rotation in $G$ is $(1,3,2,4)$. By Lemma~\ref{lemma_triangle} applied to the triangles $234$ and $134$, the edge $54$ does not cross the edges $13,23,34$ and $24$. But portions of these edges form a closed curve separating the vertices $4$ and $5$, a contradiction.
\end{proof}

\begin{lemma}\label{lemma_T6}
Let $G$ be a simple complete topological graph with vertices $1,2,\dots,7$. Suppose that $G$ contains a twisted graph $T_6$ induced by the vertices $1,2,\dots,6$, in this canonical order, and with the orientation where the rotation of the vertex $6$ is $(1,2,3,4,5)$. Then the rotation of the vertex $7$ is not $(3,2,1,4,5,6)$.
\end{lemma}

\begin{figure}
\begin{center}
\epsfbox{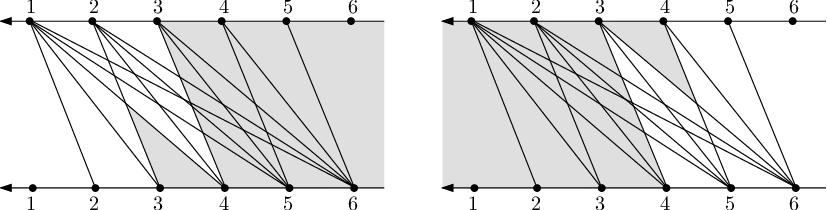} 
\end{center}
\caption{Impossibility of adding a vertex with rotation $(3,2,1,4,5,6)$ to the twisted graph $T_6$. The grey area represents the triangular face adjacent to the vertices $3$ and $4$ in the subgraphs induced by the vertices $1,2,3,4$ (left) and $3,4,5,6$ (right).}
\label{obr_6}
\end{figure}

\begin{proof}
Suppose for contradiction that the rotation of the vertex $7$ is $(3,2,1,4,5,6)$.
The subgraphs $G_1=G[\{1,2,3,4\}]$ and $G_2=G[\{3,4,5,6\}]$ are both isomorphic to the convex graph $C_4$. The $4$-cycles corresponding to the outer face of $C_4$ are $1243$ and $3465$, respectively. The two triangular faces adjacent to the vertices $3$ and $4$ in $G_1$ and $G_2$ cover the whole plane; see Figure~\ref{obr_6}. It follows that at least one of these two faces contains the vertex $7$. The rotation of the vertex $7$ is $(1,4,3,2)$ in $G[1,2,3,4,7]$ and $(3,4,5,6)$ in $G[3,4,5,6,7]$, which contradicts Lemma~\ref{lemma_C4}.
\end{proof}

\subsection{Proof of Theorem~\ref{veta_uplne}}\label{sub_dukaz_vety_uplne}
Now we finish the proof of Theorem~\ref{veta_uplne} by combining previous results from this section.
Let $g(n)$ be the number of different rotation systems of simple complete topological graphs with $n$ vertices. By Proposition~\ref{tvrz_1graf-2rot}, we have $T_w(n)\le g(n)$. We show an upper bound on $g(n)$ by induction.

Let $N=4^{30^4}$. Assume that $n\ge 2N$, otherwise $g(n)\le g(2N)$, which is a constant. We may also assume for simplicity that $n=2^k$ where $k$ is a positive integer.

Let $G$ be a simple complete topological graph with vertices $v_1,v_2,\dots,v_n$. 
Let $G_1$ be the subgraph of $G$ induced by the vertices $v_1,\dots,v_{n/2}$ and let $G_2$ be the subgraph of $G$ induced by the vertices $v_{n/2+1},\dots,v_n$. Fix a rotation system $\mathcal{R}_1$ for $G_1$ and $\mathcal{R}_2$ for $G_2$. Choose an arbitrary drawing of $G_1$ with the rotation system $R_1$.
Let $v_i$ be a vertex of $G_2$. We show an upper bound on the number of possible rotations of $v_i$ in the subgraph $G_1^i$ of $G$ induced by $V(G_1)\cup\{v_i\}$.

By Theorem~\ref{theorem_Cm1_or_Tm2}, every simple complete topological graph with $N$ vertices contains $C_5$ or $T_6$. Therefore, every induced subgraph of $G_1$ with $N$ vertices contains a subgraph $H$ weakly isomorphic to $C_5$ or $T_6$. By Lemma~\ref{lemma_C5} or~\ref{lemma_T6}, one of the cyclic permutations of the vertices of $H$ is forbidden in the rotation of $v_i$. Note that Lemmas~\ref{lemma_C5} and~\ref{lemma_T6} can be applied regardless of the particular way how $H$ is drawn.
Consequently, for each $N$-tuple of vertices in $G_1$, a non-empty subset of their cyclic permutations is forbidden in the rotation of $v_i$. 

Let $\mathcal{R}_1^i$ denote the set of all possible rotations of $v_i$ in $G_1^i$. 
To pass from cyclic permutations to linear permutations, we arbitrarily select a first element in each cyclic permutation from $\mathcal{R}_1^i$ and denote the resulting set of permutations as $\mathcal{P}_1^i$. For each forbidden cyclic permutation $\rho$ of $N$ elements, the permutations from $\mathcal{P}_1^i$ avoid all $N$ linear permutations obtained from $\rho$. In particular, the VC-dimension of the set $\{\pi^{-1}; \pi\in\mathcal{P}_1^i\}$ is at most $N-1$. Let $f(m)$ be the maximum possible size of a set of $m$-permutations with VC-dimension $N-1$. By Theorem~\ref{veta_vcdimenze}, 
$$|\mathcal{R}_1^i|=|\mathcal{P}_1^i|\le f(n/2)\le 2^{(n/2) \cdot ((2/t!)\alpha(n/2)^t + O(\alpha(n/2)^{t-1}) )},$$
where $t=(N-2)/2$.
For every $i>n/2$, the rotation of $v_i$ in $G$ is uniquely determined by the rotation $\pi_i$ of $v_i$ in $G_1^i$, the rotation $\pi'_i$ of $v_i$ in $G_2$ and by one of the $\frac{(n/2)(n/2-1)}{n-1}{n-1 \choose n/2}\le n2^n$ ways of merging $\pi_i$ and $\pi'_i$ together. For $i\le n/2$, the situation is symmetric.

It follows that the number of all possible rotation systems of $G$ with $\mathcal{R}_1$ and $\mathcal{R}_2$ fixed is at most 
\begin{align*}
(f(n/2)\cdot n2^n)^n &\le n^n\cdot 2^{n^2}\cdot 2^{(n^2/2) \cdot ((2/t!)\alpha(n/2)^t + O(\alpha(n/2)^{t-1}) )}\\
&\le 2^{c(n^2/2) \cdot \alpha(n)^t}, 
\end{align*}
where $c$ is an absolute constant.
Since there are $g(n/2)$ possibilities for each of the rotation systems $\mathcal{R}_1$ and $\mathcal{R}_2$, we have
\begin{align*}
g(n)&\le (g(n/2))^2\cdot 2^{c(n^2/2) \cdot \alpha(n)^t}\\
&\le g(2N)^{n}\cdot 2^{c(n^2/2 + 2(n/2)^2/2 + 4(n/4)^2/2 + \cdots) \cdot \alpha(n)^t}\\
&\le g(2N)^{n}\cdot 2^{c(n^2) \cdot \alpha(n)^t} = 2^{O(n^2 \cdot \alpha(n)^t)}.
\end{align*}

\subsection{Combinatorial generalization of Theorem~\ref{veta_uplne}} \label{section_combinatorial}

Here we generalize Theorem~\ref{veta_uplne} to a purely combinatorial statement involving $n$-tuples of cyclic permutations. The aim is to estimate the number of possible rotation systems of a simple complete topological graph using as little topological information as possible. In particular, the only condition we need comes from drawings of complete graphs with $4$ vertices.

\begin{observation}{\rm \cite{K09_enumeration,PT04_how}}\label{obs_8_rotacnich_systemu}
The eight rotation systems listed in Table~\ref{rotacni_systemy_K4} are the only possible rotation systems of a simple complete topological graph with vertices $1,2,3,4$. 
\end{observation}

\begin{table}
\begin{center}
\begin{tabular}{c|c}  
graph & rotation system \\
\hline
$H_1$ & $((2,4,3),(1,3,4),(1,4,2),(1,2,3))$ \\
$H^R_2$ & $((2,4,3),(1,4,3),(1,2,4),(1,2,3))$ \\
$H^R_3$ & $((2,3,4),(1,3,4),(1,2,4),(1,2,3))$ \\
$H^R_4$ & $((2,3,4),(1,4,3),(1,4,2),(1,2,3))$ \\
$H^R_1$ & $((2,3,4),(1,4,3),(1,2,4),(1,3,2))$ \\
$H_2$ & $((2,3,4),(1,3,4),(1,4,2),(1,3,2))$ \\
$H_3$ & $((2,4,3),(1,4,3),(1,4,2),(1,3,2))$ \\
$H_4$ & $((2,4,3),(1,3,4),(1,2,4),(1,3,2))$ \\
\end{tabular}
\caption{The eight possible rotation systems of a simple complete topological graph with $4$ vertices. The labels refer to the drawings in Figure~\ref{obr_3}, where $H^R_i$ denotes the mirror image of $H_i$.}
\label{rotacni_systemy_K4}
\end{center}
\end{table}

The eight rotation systems from Observation~\ref{obs_8_rotacnich_systemu} can be characterized by the following {\em parity condition}. Let $l\in\{1,2,3,4\}$ and $\{i,j,k\}=\{1,2,3,4\}\setminus\{l\}$, with $i=\min\{i,j,k\}$. We call the rotation $(i,j,k)$ at $l$ {\em positive\/} if $j<k$ and {\em negative\/} if $k<j$. A $4$-tuple of rotations at vertices $1,2,3,4$ forms a rotation system of a simple complete topological graph with the vertices $1,2,3,4$ if and only if the number of negative rotations is even.
Note that this characterization does not depend on the particular linear ordering of the vertices.

An {\em abstract rotation system\/} $\mathcal{R}$ on a set $V=\{v_1,\dots,v_n\}$ is an $n$-tuple of cyclic $(n-1)$-permutations $\pi_{v_1}, \dots, \pi_{v_n}$ where the set of elements of $\pi_{v_i}$ is $V\setminus\{v_i\}$. A {\em subsystem of $\mathcal{R}$ induced\/} by a subset $W=\{w_1,\dots,w_k\}\subset V$, denoted by $\mathcal{R}[W]$, is a $|W|$-tuple of cyclic permutations $\rho_{w_1},\dots,\rho_{w_k}$ where $\rho_{w_i}$ is a restriction of $\pi_{w_i}$ to the subset $W\setminus\{w_i\}$.

An abstract rotation system is {\em realizable\/} if it is a rotation system of a simple complete topological graph. Realizable rotation systems on a set $W$ of size $4$ are precisely those satisfying the parity condition for some linear ordering of $W$.
An abstract rotation system $\mathcal{R}$ is {\em good\/} if every subsystem of $\mathcal{R}$ induced by a $4$-element subset is realizable.

We prove the following theorem, generalizing Theorem~\ref{veta_uplne}.

\begin{theorem}\label{veta_kombinatoricka}
The number of good abstract rotation systems on an $n$-element set is at most
$$2^{n^2\cdot \alpha(n)^{O(1)}}.$$
\end{theorem}

We do not know whether the upper bound in Theorem~\ref{veta_kombinatoricka} is asymptotically tight. The best lower bound $2^{\Omega(n^2)}$ on the number of good abstract rotation systems comes from Theorem~\ref{veta_uplne_historicka}.

\begin{problem}\label{problem_good}
Is it true that the number of good abstract rotation systems on an $n$-element set is bounded by $2^{O(n^2)}$?
\end{problem}

We note that the asymptotic number of abstract rotation systems may vary significantly if a different pattern of the same size is forbidden. There are $16$ possible abstract rotation systems on every $4$-element set. If the forbidden pattern consists of a different set of eight abstract rotation systems, we may obtain $2^{\Omega(n^2\log n)}$ abstract rotation systems on $n$ elements satisfying this restriction.
For example, consider the set $\mathcal{A}$ of all abstract rotation systems on the set $\{1,2,\dots,n\}$ where in every induced subsystem on four elements $i<j<k<l$, we forbid the eight abstract rotation systems with rotation $(j,l,k)$ at $i$. The following construction shows that the size of $\mathcal{A}$ is $2^{\Omega(n^2\log n)}$. 
Consider an abstract rotation system $\mathcal{R}=(\pi_1,\pi_2,\dots,\pi_n)$ where $\pi_i(j) \in \{1,\dots,i-1\}$ for $j \le i-1$ and $\pi_i(j)=j+1$ for $j\ge i$. Clearly, the rotation at $i$ in every subsystem of $\mathcal{R}$ induced by four elements $i<j<k<l$ is $(j,k,l)$. The number of such abstract rotation systems is $\prod_{i=1}^n (i-1)!=2^{\Omega(n^2\log n)}$.

Good abstract rotation systems do not characterize realizable abstract rotation systems completely. For example, the following two good abstract rotation systems on five elements are not realizable:

\begin{align*}
\mathcal{R}^5_1 &= ((2,5,3,4),(1,3,4,5),(1,2,5,4),(1,2,5,3),(1,3,4,2)), \\
\mathcal{R}^5_2 &= ((2,3,5,4),(1,3,4,5),(1,5,2,4),(1,2,5,3),(1,4,3,2)).
\end{align*}
It is straightforward to check that both $\mathcal{R}^5_1$ and $\mathcal{R}^5_2$ are good. Suppose that these systems are realizable. In both cases, in the subgraph $H$ induced by the vertices $1,2,3,4$, the edges $13$ and $24$ cross. Fix a drawing of $H$ as a convex graph with the vertices $1,2,3,4$ on the outer face in clockwise order; see Figure~\ref{obr_7}. In both cases, the orientations of triangles and the rotations of vertices imply, by Lemma~\ref{lemma_triangle}, that the vertex $5$ must lie inside the triangles $132$ and $143$. But this is impossible as the two triangles have disjoint interiors.

\begin{figure}
\begin{center}
\epsfbox{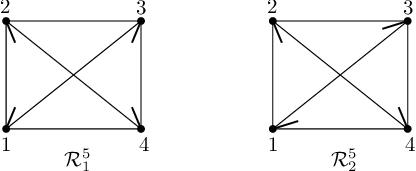} 
\end{center}
\caption{Partial realizations of the good abstract rotation systems $\mathcal{R}^5_1$ and $\mathcal{R}^5_2$. Thick segments represent the portions of the edges incident with the vertex $5$.}
\label{obr_7}
\end{figure}

While it is likely that there is no finite characterization of realizable abstract rotation systems by a finite list of forbidden subsystems, it is known that realizable abstract rotation systems can be recognized in polynomial time~\cite{K10_simple_real}.

To prove Theorem~\ref{veta_kombinatoricka}, we proceed in the same way as in the proof of Theorem~\ref{veta_uplne}, but we need to replace Theorem~\ref{theorem_Cm1_or_Tm2}, Lemma~\ref{lemma_C5} and Lemma~\ref{lemma_T6} by combinatorial analogues.

An abstract rotation system on $n$ elements is called {\em convex\/} and denoted by $\mathcal{C}_n$ if the elements can be ordered in a sequence $v_1,v_2,\dots,v_n$ so that
the rotation at $v_i$ is $(v_1,v_2\dots,v_{i-1}$, $v_{1+1}$, $v_{i+2},\dots,v_n)$. An abstract rotation system on $n$ elements is called {\em twisted\/} and denoted by $\mathcal{T}_n$ if the elements can be ordered in a sequence $v_1,v_2,\dots,v_n$ so that the rotation at $v_i$ is $(v_{i-1},\dots,v_2$, $v_{1}$, $v_{1+1}$, $v_{i+2}\dots,v_n)$. Note that $\mathcal{C}_n$ is a rotation system of the convex graph $C_n$ and $\mathcal{T}_n$ is a rotation system of the twisted graph $T_n$.

Two abstract rotation systems are {\em isomorphic\/} if they differ only by relabeling of their ground set. An abstract rotation system $\mathcal{R}$ {\em contains\/} an abstract rotation system $\mathcal{S}$ if $\mathcal{R}$ has an induced subsystem isomorphic to $\mathcal{S}$. 

The following theorem generalizes Theorem~\ref{theorem_Cm1_or_Tm2}. 

\begin{theorem}\label{veta_unavoidable_kombinatoricka}
For all positive integers $m_1, m_2$ there is an $M$ such that every good abstract rotation system on  $M$ elements contains $\mathcal{C}_{m_1}$ or $\mathcal{T}_{m_2}$. 
\end{theorem}

To keep the proof simple, we do not try to optimize the value of $M$, as a function of the parameters $m_1$ and $m_2$. However, it is likely that the same bound as in Theorem~\ref{theorem_Cm1_or_Tm2} can be proved even in this generalized setting, by adapting the original topological proof~\cite{PST03_unavoidable}. We also note that the assumption of being good is not necessary: Theorem~\ref{veta_unavoidable_kombinatoricka} holds in general for all abstract rotation systems, only with larger values of $M$.

\begin{proof}
Let $(\pi_1, \pi_2, \dots, \pi_M)$ be a good abstract rotation system on the set $\{1,$ $2,$ $\dots,$ $M\}$. Assume without loss of generality that $\pi_1=(2,3,\dots,M)$ and that $\pi_i(1)=1$ for $i>1$. 
For every three elements $i,j,k$ with $1<i<j<k$, consider the induced abstract rotation system $\mathcal{R}[\{1,i,j,k\}]$. For $l\in \{i,j,k\}$, let $t_{i,j,k}(l)=1$ if the rotation at $l$ in $\mathcal{R}[\{1,i,j,k\}]$ is positive and $t_{i,j,k}(l)=0$ if the rotation at $l$ in $\mathcal{R}[\{1,i,j,k\}]$ is negative. The {\em type\/} of the triple $(i,j,k)$ is the sequence $t_{i,j,k}(i)t_{i,j,k}(j)t_{i,j,k}(k)$. By the parity condition, we have the following four types of triples: $111,100,010$ and $001$. By Ramsey's theorem, if $M$ is sufficiently large, there is a subset $W\subseteq \{2,3,\dots,M\}$ of size $m=\max(m_1,m_2)$ such that all triples from $W$ have the same type. Without loss of generality, assume that $W=\{2,3,\dots,m+1\}$. Let $abc$, with $a,b,c\in\{0,1\}$, be the type shared by all the triples from $W$.
If $a=1$, then for each $l \in W$, the entries $l+1, l+2, \dots, m+1$ form an increasing sequence in $\pi_l$. If $a=0$, the entries $l+1, l+2, \dots, m+1$ form a decreasing sequence in $\pi_l$. Similarly, the entries $2,3,\dots,l-1$ form an increasing sequence in $\pi_l$ if $c=1$ and a decreasing sequence if $c=0$. If $b=1$, then in $\pi_l$, all entries smaller than $l$ appear before all entries larger than $l$. If $b=0$, then in $\pi_l$, all entries smaller than $l$ appear after all entries larger than $l$. Therefore, if $abc=111$ or $010$, then $W$ induces an abstract rotation system isomorphic to $\mathcal{C}_m$, and if $abc=100$ or $001$, then $W$ induces an abstract rotation system isomorphic to $\mathcal{T}_m$.
\end{proof}

The following two lemmas generalize Lemma~\ref{lemma_C5} and Lemma~\ref{lemma_T6}. Again, we do not try to optimize the sizes of the two abstract rotation systems $\mathcal{C}_{m_1}$ and $\mathcal{T}_{m_2}$.

\begin{lemma}\label{lemma_abstract_C7} 
Let $\mathcal{R}$ be a good abstract rotation system on the set $\{1,2,\dots,8\}$. Suppose that the subsystem of $\mathcal{R}$ induced by the vertices $1,2,\dots,7$ is $\mathcal{C}_7$, with $(v_1, \dots,v_7)=(1,\dots,7)$. Then the rotation at $8$ is not $(1,3,5,7,2,4,6)$. 
\end{lemma}

\begin{proof}
Let $\mathcal{R}=(\pi_1,\pi_2,\dots,\pi_8)$ and suppose for contradiction that $\pi_8 = (1, 3,$ $5, 7, 2, 4, 6)$. Let $i,i+1,i+2$ be three consecutive numbers in the cyclic sequence $(1,2,\dots,7)$. The subsystem $\mathcal{R}[\{i,i+1,i+2,8\}]=(\rho^i_i,\rho^i_{i+1},\rho^i_{i+2},\rho^i_8)$ has at least one negative triple among $\rho^i_i,\rho^i_{i+1},\rho^i_{i+2}$. If $\rho^j_j$ is negative, that is, $\rho^j_j=(j+1,8,j+2)$, we have $\pi_j=(1,2,\dots,j-1,j+1,8,j+2,\dots,7)$. Similarly, if $\rho^{j-1}_j$ is negative, then $\pi_j=(1,2,\dots,j-1,8,j+1,j+2,\dots,7)$. Finally, if $\rho^{j-2}_j$ is negative, then $\pi_j=(1,2,\dots,j-2,8,j-1,j+1,j+2,\dots,7)$. Therefore, a negative triple $\rho^i_j$ precisely determines the position of the element $8$ in the rotation $\pi_j$, and each such rotation can arise from at most one negative triple $\rho^i_j$. It follows that in each of the rotations $\pi_j, j\in \{1,2,\dots,7\}$, the element $8$ appears in one of the three possible positions between the elements $j-2$ and $j+2$. But then the subsystem
$\mathcal{R}[\{1,3,5,10\}]=((10,3,5),(1,10,5),(1,3,10),(1,3,5))$ has exactly one negative triple, a contradiction.
\end{proof}

Let $\mathcal{R}=(\rho_1,\rho_2,\rho_3,\rho_4)$ be an abstract rotation system on a $4$-element set. The {\em signature\/} of $\mathcal{R}$ is a sequence $(\varepsilon_1,\varepsilon_2,\varepsilon_3,\varepsilon_4)$ of four symbols, where $\varepsilon_i$ is '$+$' if $\rho_i$ is positive and '$-$' if $\rho_i$ is negative.

\begin{lemma}\label{lemma_abstract_T816}
Let $m=816$. Let $\mathcal{R}$ be a good abstract rotation system on the set $\{1,2,\dots,m\}$. Suppose that the subsystem of $\mathcal{R}$ induced by the vertices $1,2,\dots$, $m-1$ is $\mathcal{T}_{m-1}$, with $(v_1, \dots,v_{m-1})=(1,\dots,m-1)$. Then the rotation at $m$ is not $(1,3,\dots,m-1,2,4,\dots,m-2)$. 
\end{lemma}

\begin{proof}
Let $\mathcal{R}=(\pi_1,\pi_2,\dots,\pi_{m})$ and suppose for contradiction that $\pi_{m}=(1, 3,$ $\dots,$ $m-1, 2, 4, \dots, m-2)$. 
Let $k=8$, $W=\{2k+1,2k+2,\dots, m-4\}$ and $m'=|W|=m-2k-4$.

For $i\in W \cup \{m-3, m-2\}$, we say that a rotation $\pi_i$ is of the {\em first type\/} if the element $m$ appears in $\pi_i$ within the subinterval $(i-2k, \dots, 1, i+1)$, of the {\em second type\/} if $m$ appears in $\pi_i$ within the subinterval $(i+2, \dots, m-1,i-1)$, of the {\em third type\/} if $\pi_i=(i-1,\dots,1,i+1,m,i+2,\dots,m-1)$, and of the {\em fourth type\/} if $m$ appears in $\pi_i$ within the subinterval $(i-1, \dots, i-2k)$.
Let $W_1$ ($W_2, W_3$) be the set of those elements $i\in W$ such that $\pi_i$ is of the first (second, third) type, respectively. Let $W'_4$ be the set of those elements $i\in W \cup \{m-3, m-2\}$ such that $\pi_i$ is of the fourth type.

First we show that $|W'_4|\le 8k$. If $|W'_4|\ge 8k+1$, then at least $4k+1$ elements $i_1<i_2<\dots<i_{4k+1}$ of $W'_4$ are all odd or all even. In particular, the rotation $\rho_{m}$ in the subsystem $\mathcal{R}[\{i_1,i_{2k+1},i_{4k+1},m\}]=(\rho_{i_1},\rho_{i_{2k+1}},\rho_{i_{4k+1}},\rho_{m})$ is positive. Since the rotations $\pi_{i_1}, \pi_{i_{2k+1}}$ and $\pi_{i_{4k+1}}$ are of the fourth type, we observe that the signature of $\mathcal{R}[\{i_1,i_{2k+1},i_{4k+1},m\}]$ is $(+,+,-,+)$,
which is a contradiction with the parity condition.

Next we show that $|W_3|\le m'/2 + 3$. Suppose for contradiction that $|W_3|\ge m'/2 + 4$. Let $W_3^E$ be the set of even elements of $W_3$ and let $I$ be the smallest interval containing $W_3^E$.
Let $W_3^O=W_3\setminus W_3^E$ be the set of odd elements of $W_3$. Since $|W_3^E \cup (W_3^O\setminus I)|\le m'/2+1$, the interval $I$ contains at least $3$ odd elements $o_1<o_2<o_3$ of $W_3$. In particular, for $e_1=\min I$ and $e_3=\max I$, we have $e_1,e_3 \in W_3^E$, $o_2\ge e_1 + 3$ and $e_3 \ge o_2 + 3$. It follows that $\mathcal{R}[\{e_1,o_2,e_3,m\}]=((m,o_2,e_3),(e_1,m,e_3),(o_2,e_1,m),(o_2,e_1,e_3))$. But this subsystem has signature $(+,-,-,-)$, a contradiction.

For each $i\in W_1$ and $j\in\{1,2,\dots,k\}$, we consider the subsystem $\mathcal{R}[\{i-2j+1,i,i+1,m\}]=(\rho^{i,j}_{i-2j+1},\rho^{i,j}_i,\rho^{i,j}_{i+1},\rho^{i,j}_{m})$.
Since the parity of $i$ is opposite to the parity of $i-2j+1$ and $i+1$, the rotation $\rho^{i,j}_{m}$ is negative. Since the rotation $\pi_i$ is of the first type, we have $\rho^{i,j}_i=(i-2j+1,m,i+1)$. It follows that the signature of $\mathcal{R}[\{i-2j+1,i,i+1,m\}]$ is either $(+,-,+,-)$ or $(-,-,-,-)$. Moreover, there is a $j(i)\in\{0,1,\dots,k\}$ such that for $j\le j(i)$ the signature of $\mathcal{R}[\{i-2j+1,i,i+1,m\}]$ is $(-,-,-,-)$ and for $j> j(i)$ the signature of $\mathcal{R}[\{i-2j+1,i,i+1,m\}]$ is $(+,-,+,-)$.

Similarly for each $i\in W_2$ and $j\in\{1,2,\dots,k\}$, we consider the subsystem $\mathcal{R}[\{i-2j,i,i+2,m\}]=(\sigma^{i,j}_{i-2j},\sigma^{i,j}_i,\sigma^{i,j}_{i+1},\sigma^{i,j}_m)$. We have $\sigma^{i,j}_m=(i-2j,i,i+2)$ and $\sigma^{i,j}_i=(i-2j,i+2,m)$, thus $\mathcal{R}[\{i-2j,i,i+2,m\}]$ has signature either $(+,+,+,+)$ or $(-,+,-,+)$. Again, there is a $j(i)\in\{0,1,\dots,k\}$ such that the signature is $(-,+,-,+)$ for $j\le j(i)$ and $(+,+,+,+)$ for $j> j(i)$.

Let $W_1^{+}=\{i\in W_1; j(i)<k\}$. For every $i\in W_1^{+}$, the signature of $\mathcal{R}[\{i-2k+1,i,i+1,m\}]$ is $(+,-,+,-)$. In particular, the rotation $\pi_{i+1}$ is of the fourth type. Therefore, $|W_1^{+}|\le |W'_4| \le 8k$.

Similarly, let $W_2^{+}=\{i\in W_2; j(i)<k\}$. For every $i\in W_1^{+}$, the signature of 
$\mathcal{R}[\{i-2k,i,i+2,m\}]$ is $(+,+,+,+)$. In particular, the rotation $\pi_{i+2}$ is of the fourth type. Therefore, $|W_2^{+}|\le |W'_4| \le 8k$.

Let $W_1^{-} = W_1\setminus W_1^{+} = \{i\in W_1; j(i)=k\}$. For every $i \in W_2^{+}$ and every $j\in\{1,2,\dots,k\}$, the signature of $\mathcal{R}[\{i-2j+1,i,i+1,m\}]$ is $(-,-,-,-)$. In particular, $\pi_{i-2j+1}=(i-2j,\dots,1,i-2j+2,\dots,i,m,i+1,\dots,m-1)$.
Observe that for every $l\in\{2,\dots,m-5\}$, there is at most one pair $i,j$ such that $i\in W_1^{-}, j,\in\{1,2,\dots,k\}$ and $l=i-2j+1$. Thus we have $|W_1^{-}|\le \frac{m-6}{k}$.

Let $W_2^{-} = W_2\setminus W_2^{+} = \{i\in W_2; j(i)=k\}$.
For every $i \in W_2^{+}$ and every $j\in\{1,2,\dots,k\}$, the signature of $\mathcal{R}[\{i-2j,i,i+2,m\}]$ is $(-,+,-,+)$. In particular, the element $m$ appears in $\pi_{i-2j}$ in one of the two positions in the subinterval $(i,i+1,i+2)$.
This implies that for every $l\in\{1,2,\dots,m-6\}$, there is at most one pair $i,j$ such that $i\in W_2^{-}, j,\in\{1,2,\dots,k\}$ and $l=i-2j$. Thus we have $|W_2^{-}|\le \frac{m-6}{k}$.

Putting all the estimates together, we have 
\begin{align*}
m'=|W| &\le |W_1^+| + |W_1^-| + |W_2^+| + |W_2^-| + |W_3| + |W'_4|\\
 &\le \frac{m'}{2} + 3 + \frac{2(m-6)}{k} + 24k
\end{align*}
and thus
\begin{align*}
k(m-2k-4) &\le 6k + 4(m-6) + 48k^2, \\
(k-4)m    &\le 50 k^2 + 10 k - 24. \\
\end{align*}

By our choice $m=816$ and $k=8$, this gives $4\cdot 816 \le 3256$ and we have a contradiction.
\end{proof}

\subsection{Graphs with maximum number of crossings}\label{sub_max_crossings}

Harborth and Mengersen~\cite{HM92_drawings} investigated simple complete topological graphs on $n$ vertices with maximum number of crossings, which is ${n \choose 4}$. They showed the lower bound $e^{\Omega(\sqrt{n})}$ on the number $T_\mathrm{w}^{\mathrm{max}}(n)$ of different weak isomorphism classes of such (unlabeled) graphs. Their construction actually gives a better lower bound $T_\mathrm{w}^{\mathrm{max}}(n) \ge 2^{n(\log{n}-O(1))}$~\cite{K09_enumeration}.

We do not have any better upper bound on $T_\mathrm{w}^{\mathrm{max}}(n)$ than that from Theorem~\ref{veta_uplne}, thus the problem of determining $T_\mathrm{w}^{\mathrm{max}}(n)$ asymptotically seems to be wide open. However, the following observation could help with improving the upper bound to $2^{O(n^2)}$.

Let $G$ be a simple complete topological graph with vertex set $V$ and with ${|V| \choose 4}$ crossings. Let $v\in V$ and let $G'$ be a subgraph of $G$ induced by $V\setminus \{v\}$. A {\em face\/} of $G'$ is a connected region of the set obtained from the plane by removing all the edges of $G'$. 
Two faces $F'_1$ and $F'_2$ in two simple complete topological graphs $G'_1$ and $G'_2$ weakly isomorphic to $G'$ are considered {\em equivalent\/} if every triangle $T_1$ in $G'_1$ and the corresponding triangle $T_2$ in $G'_2$ satisfy the following condition: the triangles $T_1$ and $T_2$ have the same orientation if and only if either $T_1$ contains $F'_1$ and $T_2$ contains $F'_2$, or $F'_1$ is outside $T_1$ and $F'_2$ is outside $T_2$.
By a {\em combinatorial face\/} we mean an equivalence class of faces, but also any particular face from the class.
Lemma~\ref{lemma_triangle} implies that the combinatorial face of $G'$ that contains $v$ uniquely determines the rotation of $v$ in $G$. Therefore, the number of possible rotations of $v$, with the weak isomorphism class of $G'$ fixed, is bounded from above by the number $f(G')$ of possible combinatorial faces in a simple topological graph weakly isomorphic to $G'$. The number $f(G')$ may be exponential, for example when $G'$ is the convex graph $C_n$. This graph has $n/2$ pairwise crossing edges (main diagonals), which may be drawn through a common point $x$. Then each of the edges can be redrawn to go around $x$ from the left or from the right. Each of these choices produces a different combinatorial face containing $x$. On the other hand, it can be shown that $f(C_n)=2^{O(n)}$, since each of the bounded combinatorial faces of $C_n$ can be assigned to a unique subset of pairwise crossing diagonals, in the following way. Let $C$ be the Hamiltonian cycle of $C_n$ bounding the outer face. To each diagonal $e$ of $C$ we assign the region $r(e)$ bounded by $e$ and by the shorter arc of $C$ determined by the endpoints of $e$. (For the main diagonals, we choose the ``shorter'' arc arbitrarily.) Each face $f$ is assigned to a set $R(f)$ of minimal regions $r(e)$ containing $f$. The set $R(f)$ determines all triangles containing $f$, and all diagonals $e$ such that $r(e)\in R(f)$ are pairwise crossing. A set of pairwise crossing diagonals in $C_n$ is uniquely determined by the set of their endpoints. Therefore, there are at most $2^{n-1}$ possible sets $R(f)$. Accounting for two possible orientations of the drawing of $C_n$, we get the upper bound $f(C_n)\le 2^{n}+2$.

We do not know whether similar upper bound holds for all simple complete topological graphs.

\begin{problem}\label{problem_pocet_sten}
Is it true that for every simple complete topological graph $G$ with $n$ vertices, the number of possible combinatorial faces in simple complete topological graphs weakly isomorphic to $G$ satisfies $f(G)\le 2^{O(n)}?$
\end{problem}

A positive answer to Problem~\ref{problem_pocet_sten} would imply that $T_\mathrm{w}^{\mathrm{max}}(n) = 2^{O(n^2)}$, by the proof in Subsection~\ref{sub_dukaz_vety_uplne}.

A similar question can be asked in the combinatorial setting. In a simple complete topological graph with $n$ vertices and ${n \choose 4}$ crossings, every $4$-tuple of vertices induces a crossing. Therefore, for every complete subgraph with $4$ vertices there are $6$ possible rotation systems, corresponding to the rotation systems of the graphs $H_2, H_3, H_4, H^R_2, H^R_3, H^R_4$ in Table~\ref{rotacni_systemy_K4}. In addition to the parity condition, these rotation systems satisfy the following condition. There exists a pair $i,j\in\{1,2,3,4\}$ such that for $\{k,l\}=\{1,2,3,4\}\setminus \{i,j\}$, the rotation at $k$ is $(i,j,l)$ and the rotation at $l$ is $(i,j,k)$. In fact, there are always four such pairs $i,j$, corresponding to the four edges without crossing in the drawing.

\begin{problem}
What is the number of abstract rotation systems on $n$ elements, where every subsystem induced by $4$ elements is realizable as a rotation system of a simple drawing of $K_4$ with one crossing? 
\end{problem}

We do not know better lower bound than that implied by the topological construction by Harborth and Mengersen~\cite{HM92_drawings,K09_enumeration}. The best upper bound comes from Theorem~\ref{veta_kombinatoricka}.


\section{The upper bound in Theorem~\ref{veta_hlavni}}
Let $G=(V,E)$ be a graph with $n$ vertices and $m$ edges. If $v$ is an isolated vertex in $G$, then $T_w(G)=T_w(G-v)$. Thus, we may assume that $G$ has no isolated vertices. The upper bound on $T_w(G)$ for other graphs $G$ then directly follows.

Let $\mathcal{G}$ be a simple topological graph realizing $G$. A {\em topological component\/} of $\mathcal{G}$ is a maximal connected subset of the plane that is a union of vertices and edges of $\mathcal{G}$. Note that a topological component of $\mathcal{G}$ is a union of components of $G$.
A topological graph $\mathcal{G}$ is {\em topologically connected\/} if it has only one topological component.

First we extend $G$ by adding edges connecting the topological components of $\mathcal{G}$ as follows. Let $\mathcal{C}_1$ and $\mathcal{C}_2$ be two topological components of $\mathcal{G}$. We redraw $\mathcal{C}_2$ so that it has a vertex $v_2$ on the boundary of its outer face, and place this drawing inside a face of $\mathcal{C}_1$ containing a vertex $v_1$ on its boundary. Then we may add the edge $v_1v_2$ as a curve without crossings. We repeat this process until there is only one topological component. Since the graph $G$ had no isolated vertices, we added at most $n/2\le m$ new edges, so the new graph has $n$ vertices and $\Theta(m)$ edges. In this way, we might have created at most $n^{n}\le 2^{O(n\log n)}\le 2^{O(m\log n)}$ different graphs.
Thus, for proving the upper bound on $T_w(G)$, we may assume that $\mathcal{G}$ is topologically connected.

Ideally, we would like to extend the graph $G$ to a connected graph, but it is not clear that it is always possible to connect two components of $G$ that form a single topological component in the drawing by an edge so that the resulting drawing is still a simple topological graph. For example, there are simple topological graphs where some pairs of vertices from different components cannot be connected by an edge, so that the resulting drawing is still simple; see Figure~\ref{obr_8}, left.

\begin{figure}
\begin{center}
\epsfbox{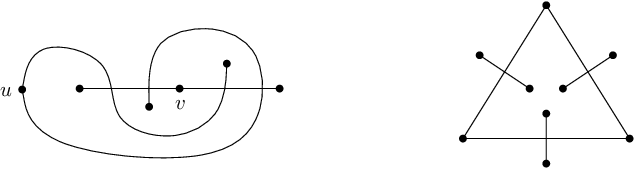} 
\end{center}
\caption{Left: A simple drawing of $P_3 + P_3$ which cannot be extended by an edge $uv$. Right: A topologically connected drawing of a graph with four components, with every spanning forest topologically disconnected.}
\label{obr_8}
\end{figure}

\subsection{A construction of a topological spanning tree}
Next we construct a {\em topological spanning tree\/} $\mathcal{T}$ of $\mathcal{G}$; see Figure~\ref{obr_9_T_representation}, left. A topological spanning tree $\mathcal{T}$ of $\mathcal{G}$ is a simply connected subset of the single topological component of $\mathcal{G}$ containing all vertices of $\mathcal{G}$ and satisfying the property that the only nonseparating points of $\mathcal{T}$ are the vertices of $\mathcal{G}$. Our goal is to find such a tree consisting of $O(n)$ connected portions of edges of $\mathcal{G}$. If $G$ is a complete graph, we may take as $\mathcal{T}$ the star consisting of all edges incident with one vertex of $\mathcal{G}$~\cite{K09_enumeration}, since such edges are internally disjoint. If $G$ is connected, we may start with a drawing of an arbitrary spanning tree of $G$, but as some edges of the tree may cross, we may need to remove portions of some edges to break cycles. If $G$ has multiple components, the construction is a bit more involved. For example, it is not enough to take a union of spanning trees of the individual components, as some of the spanning trees may be topologically disjoint, even if $\mathcal{G}$ is topologically connected; see Figure~\ref{obr_8}, right. Also we may need to include in $\mathcal{T}$ multiple disjoint portions of the same edge.

Let $C_1, \dots, C_k$ be the connected components of $G$. We choose their order in such a way that for every $i\in\{1,2,\dots,k\}$, the drawing of $C_1 \cup \dots \cup C_i$ is topologically connected.
Then for every $i\in\{2,3,\dots,k\}$, there is an edge $e_i$ in $C_i$ that crosses some edge $f_i \in C_1 \cup \dots \cup C_{i-1}$. Let $T_1$ be a spanning tree of $C_1$ and let $e_1$ be an edge of $T_1$. For every $i\in\{2,3,\dots,k\}$, let $T_i$ be a spanning tree of $C_i$ containing $e_i$. For every $i\in\{1,2,\dots,k\}$, let $e_{i,1}=e_i$ and let $e_{i,2}, \dots, e_{i,m_i}$ be the remaining edges of $T_i$ ordered in such a way that for every $j\in\{1,2,\dots,m_i\}$, the subgraph of $T_i$ formed by the edges $e_{i,1},e_{i,2}, \dots, e_{i,j}$ is connected.

In the rest of this section we often identify the vertices, edges and subgraphs of $G$ with the corresponding vertices, edges and subgraphs of $\mathcal{G}$.

The construction of $\mathcal{T}$ proceeds in $k$ phases. In the first phase, we construct a topological spanning tree $\mathcal{T}_1$ of $C_1$, in the following way. We start with the tree $\mathcal{T}_{1,1}$ consisting of the single edge $e_1$. Let $j\in \{2,3,\dots,m_1\}$ and suppose that the tree $\mathcal{T}_{1,j-1}$ has been defined. Let $v_{1,j}$ be the vertex of $e_{i,j}$ that is not contained in the edges $e_1,\dots,e_{j-1}$. If $e_{i,j}$ crosses none of the edges $e_{1,1},\dots,e_{1,j-1}$, then let $\mathcal{T}_{1,j}=\mathcal{T}_{1,j-1}\cup e_{1,j}$. Otherwise, among the crossings of $e_{1,j}$ with the edges $e_{1,1},\dots,e_{1,j-1}$, let $x_{i,j}$ be the crossing closest to $v_{1,j}$. The tree $\mathcal{T}_{1,j}$ is now obtained from $\mathcal{T}_{1,j-1}$ by attaching the portion of $e_{1,j}$ between $x_{1,j}$ and $v_{1,j}$. Finally, we put $\mathcal{T}_1=\mathcal{T}_{1,m_1}$.

Let $i\in\{2,3,\dots,k\}$ and suppose that the tree $\mathcal{T}_{i-1}$ has been defined. In the $i$th phase, we construct the tree $\mathcal{T}_{i}$ in the following way. Let $e_i=w_iw'_i$ and let $x_i$ be the crossing of $e_i$ with $f_{i-1}$. If $e_i$ crosses $\mathcal{T}_{i-1}$ in at least one point, then let $x_{i,1}$ and $x'_{i,1}$ be the crossings of $e_i$ with $\mathcal{T}_{i-1}$ closest to $w_i$ and $w'_i$, respectively. The tree $\mathcal{T}_{i,1}$ is then obtained from $\mathcal{T}_{i-1}$ by attaching
the portion of $e_i$ between $w_i$ and $x_{i,1}$ and the portion of $e_i$ between $w'_i$ and $x'_{i,1}$. If $e_i$ is disjoint with $\mathcal{T}_{i-1}$, then we construct $\mathcal{T}_{i,1}$ from $\mathcal{T}_{i-1}$ by adding the whole edge $e_i$ and joining $e_i$ with $\mathcal{T}_{i-1}$ by the shortest portion of $f_{i-1}$ connecting $x_i$ with a point of $\mathcal{T}_{i-1}$, which may be an endpoint of $f_{i-1}$ or a crossing. The rest of the $i$-th phase is similar to the construction of $\mathcal{T}_1$. In $j$-th step, we construct $\mathcal{T}_{i,j}$ from $\mathcal{T}_{i,j-1}$ by attaching the portion of $e_{i,j}$ connecting the vertex of $e_{i,j}$ not contained in $\mathcal{T}_{i,j-1}$ with the closest point of $\mathcal{T}_{i,j-1}$ along $e_{i,j}$. Finally, we put $\mathcal{T}_i=\mathcal{T}_{i,m_i}$ and $\mathcal{T}=\mathcal{T}_{k}$.

It follows from the construction that the tree $\mathcal{T}$ has $n'\le 2n$ vertices, which are either vertices or crossings of $\mathcal{G}$, and hence at most $2n$ edges, which are portions of edges of $\mathcal{G}$.

\begin{figure}
\begin{center}
\epsfbox{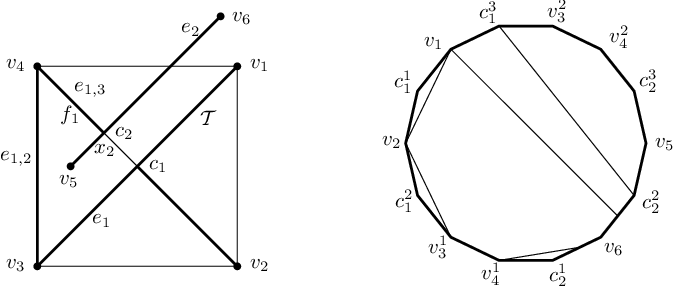}
\end{center}
\caption{A topological spanning tree $\mathcal{T}$ of a simple topological graph with two components (left) and the corresponding $\mathcal{T}$-representation (right).}
\label{obr_9_T_representation}
\end{figure}

\subsection{A construction of a $\mathcal{T}$-representation}
Now we construct the {\em $\mathcal{T}$-representation of $\mathcal{G}$}, which generalizes the star-cut representation defined in~\cite{K09_enumeration}. Consider $\mathcal{G}$ drawn on the sphere $S^2$ and cut the sphere along the edges of $\mathcal{T}$. The resulting open set $S^2\setminus \mathcal{T}$ can be mapped by an orientation preserving homeomorphism $\Phi$ to an open regular $(2n'-2)$-gon $D$, in such a way that the inverse map $\Phi^{-1}$ can be continuously extended to
the closure of $D$ so that the vertices and edges of $D$ are mapped to the vertices and edges of $\mathcal{T}$. Note that every edge of $\mathcal{T}$ corresponds to two edges of $D$, and a vertex of degree $d$ in $\mathcal{T}$ corresponds to $d$ vertices of $D$. See Figure~\ref{obr_9_T_representation}, right. During the cutting operation, every edge $e$ of $G$ can be cut into at most $n'$ pieces by the edges of $\mathcal{T}$. Each such piece becomes a {\em pseudochord\/} of $D$. That is, a simple curve in $D$ with endpoints on the boundary of $D$, and with the property that every two such curves cross in at most one point. Moreover, two pseudochords sharing an endpoint are internally disjoint, as they correspond to portions of edges with a common vertex. To separate the endpoints of the pseudochords, we cut a small disc around
each vertex $w$ of $D$, draw a part of its boundary inside $D$ as an arc $g_w$ and shorten the pseudochords incident with $w$ so that their endpoints are on $g_w$. For an edge $e$ of $D$, let $O_e$ be the counter-clockwise order of the endpoints of the pseudochords along $e$. Similarly, for each vertex $w$ of $D$, let $O_w$ be the counter-clockwise order of the endpoints of the pseudochords along $g_w$. The orders $O_e$ and $O_w$ are given as sequences of labels of the pseudochords. The collection of the orders $O_e$ and $O_w$, which together form a cyclic sequence of endpoints of the pseudochords along the boundary of $D$, is called the {\em perimetric order\/}.

The {\em $\mathcal{T}$-representation of $\mathcal{G}$} is given by (1) the topological spanning tree $\mathcal{T}$ and (2) the perimetric order $O_D$. The tree $\mathcal{T}$ is given as an abstract graph with a rotation system, which determines its combinatorial planar embedding. 

Note that the perimetric order determines which pairs of pseudochords cross and how the pseudochords connect to the edges. Thus the $\mathcal{T}$-representation of $\mathcal{G}$ determines the weak isomorphism class of $\mathcal{G}$. However, topological graphs weakly isomorphic to $\mathcal{G}$ may have several different $\mathcal{T}$-representations, which differ by the orders of crossings along the edges of $\mathcal{T}$. We say that two $\mathcal{T}$-representations are {\em weakly isomorphic\/} if they are representations of weakly isomorphic topological graphs.

\subsection{Counting topological spanning trees}
The upper bound on $T_w(G)$ will follow from an upper bound on the number of weak isomorphism classes of $\mathcal{T}$-representations of simple drawings of $G$. First we estimate the number of different topological spanning trees.

\begin{lemma}\label{lemma_trees}
Let $G$ be a graph with $n$ vertices, $m$ edges and no isolated vertices. Topologically connected simple realizations of $G$ have at most $2^{O(n\log n)}$ different topological spanning trees, up to a homeomorphism of the plane.
\end{lemma}

\begin{proof}
Let $k$ be the number of connected components of $G$. A component with $n_i$ vertices has at most $n_i^{n_i-2}$ spanning trees, hence $G$ has at most $2^{O(n\log n)}$ spanning forests. Let $T_1 \cup T_2 \cup \dots \cup T_k$ be a fixed spanning forest of $G$. The inductive construction of the topological spanning tree $\mathcal{T}$ consists of $n-k$ steps. In each step, an edge of some spanning tree $T_i$ is added to the construction. Consider the step where a portion of the edge $e_{i,j}$ is added to the tree $\mathcal{T}_{i,j-1}$. The new edge is attached either to a vertex of $\mathcal{T}_{i,j-1}$ or to an interior point of some edge of $\mathcal{T}_{i,j-1}$. There are two ways how to attach a new edge to an edge of $\mathcal{T}_{i,j-1}$, and $d$ ways how to attach a new edge to a vertex of degree $d$ in $\mathcal{T}_{i,j-1}$. Together, there are $4(n_{i,j}-1)\le 4n'-4 \le 8n$ different ways how to attach a new edge, where $n_{i,j}$ is the number of vertices of $\mathcal{T}_{i,j-1}$, and there are at most $m$ choices for the edge $e_{i,j}$. 

Now consider the step where portions of the edge $e_i$ are added to the tree $\mathcal{T}_{i-1}$. If $e_i$ crosses $\mathcal{T}_{i-1}$, then two portions of $e_i$ are added and this step is equivalent to two previous steps. If $e_i$ does not cross $\mathcal{T}_{i-1}$, then the whole edge $e_i$ and a portion of $f_{i-1}$ are added. There are at most $m$ choices for $e_i$, $m$ choices for $f_{i-1}$, two ways how to attach the portion of $f_{i-1}$ to $e_i$ and at most $8n$ different ways how to attach the portion of $f_{i-1}$ to $\mathcal{T}_{i-1}$. Altogether, we have at most $(8nm)^{n-1}\le 2^{O(n\log n)}$ ways how to construct $\mathcal{T}$.
\end{proof}

\subsection{Counting $\mathcal{T}$-representations}\label{sub_pocitani_T_reprezentaci}
It remains to estimate for each topological spanning tree $\mathcal{T}$, the maximum number of weak isomorphism classes of $\mathcal{T}$-representations. This will be the dominant term in the estimate of $T_w(G)$. 
Every edge of $G$ corresponds to at most $2n$ pseudochords in the $\mathcal{T}$-representation. Hence the $\mathcal{T}$-representation has at most $2mn$ pseudochords, with at most ${4mn \choose 8n}(4mn)!\le 2^{O(mn\log n)}$ different perimetric orders. This gives a trivial $2^{O(mn\log n)}$ upper bound on the number of weak isomorphism classes of $\mathcal{T}$-representations.

To determine the weak isomorphism class, we do not need the whole information given by the perimetric order. In fact, we only need to know the number of pseudochords corresponding to each edge of $G$ and the {\em type\/} of each pseudochord~\cite{K10_simple_real}, which we define in the next paragraph. There are at most $(2n)^m\le 2^{O(m\log n)}$ choices of the numbers of pseudochords corresponding to the edges of $G$ in the $\mathcal{T}$-representation. This upper bound is asymptotically dominated by the upper bounds in Theorem~\ref{veta_hlavni}, hence we consider these numbers fixed in the rest of this section.

The {\em type\/} $t(p)$ of a pseudochord $p$ 
is the pair $(X,Y)$ where each of $X,Y$ is either an edge of the polygon $D$ containing the endpoint of $p$ or an endpoint of $p$ on the arc $g_w$ for some vertex $w$ of $D$. For each vertex $w$ of $D$ representing a vertex $v$ of $G$, we consider $\deg(v)$ points on $g_w$ as possible values of $X$ and $Y$. For each triple of vertices $w_1,w_2,w_3$ of $D$ representing a crossing $x$ of $\mathcal{G}$, we have exactly one possible endpoint as a possible value of $X$ and $Y$, on exactly one of the arcs $g_{w_1},g_{w_2},g_{w_3}$. This follows from the fact that $\mathcal{T}$ contains exactly three portions of edges incident with $x$ and only the fourth portion becomes a pseudochord.

Let $p$ and $p'$ be pseudochords with types $(X,Y)$ and $(X',Y')$, respectively. We say that the types $(X,Y)$ and $(X',Y')$ are 
\begin{description}
\item {\em crossing\/} if the elements $X,X',Y,Y'$ are pairwise distinct and their cyclic order around the boundary of $D$ is $(X,X',Y,Y')$ or $(X,Y',Y,X')$,
\item {\em avoiding\/} if they are not crossing and all the elements $X,X',Y,Y'$ are pairwise distinct,
\item {\em parallel\/} if $(X,Y)=(X',Y')$ or $(X,Y)=(Y',X')$, and 
\item {\em adjacent\/} otherwise, that is, if exactly one of the following four equalities holds: $X=X'$, $X=Y'$, $Y=X'$ or $Y=Y'$.
\end{description}

See Figure~\ref{obr_10_typy} for examples.

\begin{figure}
\begin{center}
\epsfbox{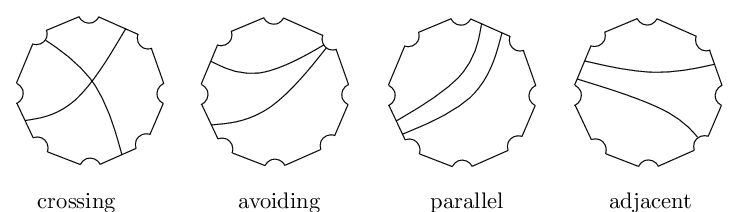}
\end{center}
\caption{Four categories of pairs of types of pseudochords.}
\label{obr_10_typy}
\end{figure}

If the elements $X,Y,X',Y'$ are pairwise distinct, we can directly determine whether $p$ and $p'$ cross: crossing types imply crossing pseudochords and avoiding types imply disjoint pseudochords. If, for example, $X=X'$ (in which case $X$ is an edge of $D$), we cannot determine whether $p$ and $p'$ cross, since this depends on the relative position of the endpoints of $p$ and $p'$ on $X$. The pairs of pseudochords with parallel and adjacent types can be arranged into maximal sequences, called {\em ladders}, formed by portions of two edges of $G$, for which we can determine whether they cross or not. See~\cite{K10_simple_real} for details.

A pseudochord is called {\em external\/} if it represents the initial or the terminal portion of an edge of $G$. Thus, at least one of the endpoints of an external pseudochord lies on one of the arcs $g_w$ where $w$ is a vertex of $D$ representing a vertex of $G$. All the other pseudochords are called {\em internal}. Every external pseudochord can have one of $O((n+m)^2)$ possible types. Every internal pseudochord, representing an internal portion of an edge of $G$, can have only $O(n^2)$ different types, since for the variables $X,Y$, we are considering only edges of $D$ and points on the arcs $g_w$, where $w$ is a vertex of $D$ representing a crossing of $\mathcal{G}$. Altogether, there are at most $(O(n+m)^{4m})\le 2^{O(m\log n)}$ combinations of types of the external pseudochords. This is again asymptotically dominated by the upper bounds in Theorem~\ref{veta_hlavni}. In the rest of this section, we consider only internal pseudochords. For a subset $F\subseteq E$ of edges of $G$, let $f(F)$ be the number of possible combinations of types of the internal pseudochords corresponding to the edges from $F$. Similarly, for a set $S$ of internal pseudochords, let $f(S)$ be the number of possible combinations of types of pseudochords from $S$. Our goal is to obtain a good upper bound on $f(E)$.

A trivial estimate gives the upper bound $f(E)\le O(n^2)^{mn}=2^{O(mn\log n)}$. This can be improved by considering the fact that the pseudochords representing a common edge of $G$ do not cross. Also note that for two pseudochords $p,p'$ representing a common edge $e$, their types 
are always avoiding. It follows that the set of types of the pseudochords representing $e$ can be represented as a noncrossing matching of size at most $2n$ on a set of at most $8n$ points in convex position, where each point corresponds to an edge or a vertex of $D$. Observe that the order of the pseudochords along $e$ can be reconstructed from this matching, thus this representation is injective. The number of such matchings is bounded from above by $2^{O(n)}$. 
Together, this gives the upper bound $f(E)\le 2^{O(nm)}$.

This estimate can be improved even further. In a simple topological graph, edges incident to a common vertex $v$ do not cross. Therefore, all the internal pseudochords representing edges incident with $v$ are pairwise disjoint. Let $P(v)$ be the set of these pseudochords. Note that two pseudochords from $P(v)$ representing different edges may have avoiding, parallel or adjacent types. Let $d$ be the degree of $v$. Similarly as before, we can represent the set of types of the pseudochords from $P(v)$ as a noncrossing matching $M$ on a set of at most $8dn$ points in convex position, where each vertex of $D$ is represented by a point and each edge of $D$ is represented by $d$ consecutive points. Again, from the matching $M$ and from the types of the external pseudochords representing the edges incident with $v$ we can uniquely determine which edge each pseudochords represents and how the pseudochords connect together to form the (portions of) edges incident with $v$. A straightforward upper bound $f(P(v))\le 2^{O(dn)}$ follows. To get a better upper bound, we observe that many of these pseudochords share the same type. More precisely, we have up to $2dn$ pseudochords in $P(v)$, but only $O(n)$ different types, since no two of the types are crossing. There are $2^{O(n)}$ ways of choosing the set of pairwise noncrossing types for internal pseudochords. For a fixed set $S$ of $O(n)$ types, we assign to each type $t\in S$ its {\em weight}, that is, a positive integer $n(t)$ denoting the number of pseudochords from $P(v)$ with type $t$. The set $\{n(t), t\in S\}$ satisfying the property $\sum_{t\in S} n(t)=|P(v)|$ is called the {\em weight vector\/} of $S$. From the set $S$ and its weight vector, we can reconstruct the matching $M$ and determine the type of each pseudochord and how the pseudochords connect to edges. This idea is similar to encoding curves on a surface using normal coordinates~\cite{SSS02_normal,SSS03_recognizing}. For a fixed $S$, there are ${O(dn)\choose O(n)}=d^{O(n)}=2^{O(n\log d)}$ different weight vectors. This gives the upper bound $f(P(v))\le 2^{O(n\log d)}$. By Jensen's inequality, $f(E)\le 2^{O(n^2\log (m/n))}$. 
Together with Lemma~\ref{lemma_trees}, this gives the first upper bound in Theorem~\ref{veta_hlavni}.

The previous method gives a good upper bound on $T_w(G)$ for dense graphs. For graphs with $o(n^2)$ edges, the method is useful if the graph has very irregular degree sequence; more precisely, if it has a small number of vertices covering almost all the edges. For graphs with $o(n^{3/2})$ edges and with most of the vertices of degree $\Theta(m/n)$, we get better results by considering larger subsets of edges. We just need to balance the number of edges in the subset to keep the number of their crossings small enough.

\begin{lemma}\label{lemma_ktice_hran}
Let $F\subseteq E$ be a set of $k$ edges. Then
$$f(F)\le {O(m+k^2) \choose O(k^2)}\cdot 2^{O(k^2\log k)} \cdot 2^{O(n+k^2)} \cdot {kn\choose O(n+k^2)}.$$
In particular, for $k=\lfloor \sqrt{n} \rfloor$ we have
$$f(F)\le 2^{O(n\log n)}.$$
\end{lemma}

\begin{proof}
Let $P(F)$ be the set of (both external and internal) pseudochords representing the edges of $F$. Since every two edges cross at most once, there are at most ${k \choose 2}$ crossings among the pseudochords from $P(F)$. In particular, at most $k^2$ pseudochords from $P(F)$ cross other pseudochord from $P(F)$. Let $P_1(F)\subseteq P(F)$ be the set of pseudochords crossing at least one pseudochord from $P(F)$. Let $P_0(F)$ be the set of internal pseudochords from $P(F)\setminus P_1(F)$. We estimate the number of perimetric orders of $|P_0(F) \cup P_1(F)|$ pseudochords in $D$ inducing at most ${k \choose 2}$ crossings. Each such perimetric order, together with the set of types of the external pseudochords from 
$P(F)$, determine the types of all pseudochords from $P(F)$, since no member of $P(F) \setminus P_1(F)$  crosses a member of $P_0(F) \cup P_1(F)$.

For the pseudochords from $P_1(F)$, we have at most ${O(m+k^2) \choose 2k^2}$ ways of choosing the set of their endpoints on the boundary of $D$, and at most $(k^2)!\le 2^{O(k^2\log k)}$ ways of matching them together. Here we do not need to optimize for matchings inducing $O(k^2)$ crossings. However,  Proposition~\ref{prop_chord_diagrams} in the next section implies the upper bound $2^{O(k^2)}$. 

The pseudochords from $P_0(F)$ form a noncrossing matching in the regions of $D\setminus (\bigcup P_1(F))$. To determine the positions of the pseudochords from $P_0(F)$, we need to refine their types into {\em subtypes\/} by splitting the edges of $D$ by the endpoints of the pseudochords from $P_1(F)$. See Figure~\ref{obr_11_subtypy}. There are at most $O(n+k^2)$ portions of edges of $D$ after this splitting, hence at most $2^{O(n+k^2)}$ choices for the set of pairwise noncrossing subtypes of the pseudochords from $P_0(F)$. Finally, there are at most ${kn\choose O(n+k^2)}$ ways of assigning a vector of positive integers with total sum at most $kn$ to the chosen set of subtypes. This is sufficient to determine the perimetric order of the pseudochords from $P(F)$ and the lemma follows.
\end{proof}

\begin{figure}
\begin{center}
\epsfbox{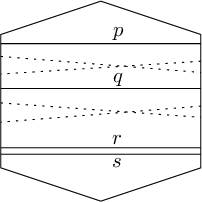}
\end{center}
\caption{Illustration for the proof of Lemma~\ref{lemma_ktice_hran}. The dotted lines represent pseudochords from $P_1(F)$. The pseudochords $p,q,r,s$ all have the same type, $r$ and $s$ also have the same subtype, but $p,q$ and $r$ have pairwise different subtypes.}
\label{obr_11_subtypy}
\end{figure}

The second upper bound in Theorem~\ref{veta_hlavni} is proved as follows. By Lemma~\ref{lemma_trees}, we may fix a topological spanning tree. Then we partition the edge set of $G$ into $O(m/\sqrt{n})$ subsets of size at most $\sqrt{n}$ and apply Lemma~\ref{lemma_ktice_hran} to each of the subsets. 
Theorem~\ref{veta_pseudosegmenty} is a special case of Theorem~\ref{veta_hlavni}, where the graph $G$ is a matching.


\section{The upper bound in Theorem~\ref{veta_neizo}}\label{section_horni_strong}

We start with some additional definitions and a combinatorial definition of the isomorphism of topological graphs. Then we show that we need to consider only topologically connected topological graphs. Finally, we reduce the problem to counting isomorphism classes of arrangements of pseudochords and present two different solutions to this problem. In the first solution we split the problem into two parts: enumerating chord diagrams and enumerating arrangements with fixed boundary, using encoding by binary vectors. The second approach is based on enumerating the dual graphs of the arrangements, which form a subclass of quadrangulations of a disc.  

\subsection{A combinatorial definition of isomorphism}\label{sub_5_1_combinatorial_izomorphism}

A {\em rotation\/} of a crossing $c$ in a topological graph is the clockwise cyclic order in which the four portions of the two edges crossing at $c$ leave the point $c$. Note that each crossing has exactly two possible rotations.
An {\em extended rotation system\/} of a simple topological graph is the set of rotations of all its vertices and crossings. 
Assuming that $T$ and $T'$ are drawings of the same abstract graph, we say that their (extended) rotation systems are {\em inverse\/} if for each vertex $v \in V(T)$ (and each crossing $c$ in $T$) the rotation of $v$ and the rotation of the corresponding vertex $v' \in V(T')$ are inverse cyclic permutations (and so are the rotation of $c$ and the rotation of the corresponding crossing $c'$ in $T'$). For example, if $T'$ is a mirror image of $T$, then $T$ and $T'$ have inverse (extended) rotation systems. 

Topologically connected topological graphs $G$ and $H$ are {\em isomorphic\/} if (1) $G$ and $H$ are weakly isomorphic, (2) for each edge $e$ of $G$ the order of crossings with the other edges of $G$ is the same as the order of crossings on the corresponding edge $e'$ in $H$, and (3) the extended rotation systems of $G$ and $H$ are either the same or inverse. This induces a one-to-one correspondence between the faces of $G$ and $H$ such that the crossings and the vertices incident with a face $f$ of $G$ appear along the boundary of $f$ in the same (or inverse) cyclic order as the corresponding crossings and vertices in $H$ appear along the boundary of the face $f'$ corresponding to $f$. It follows from Jordan--Sch\"onflies theorem that this definition is equivalent to the previous one in Section~\ref{section_intro}.

Let $G$ be a topological graph with more than one topological component. The {\em face structure\/} of $G$ is a collection of face boundaries, represented as oriented facial walks in the underlying abstract graph, of all noncontractible faces of $G$, that is, faces with more than one boundary component. The orientations are chosen in such a way that either for each noncontractible face the facial walk of the outer boundary component is oriented clockwise and the facial walks of all inner boundary components are oriented counter-clockwise, or vice versa. Both choices are regarded as giving the same face structure. By this condition, the orientations of the facial walks in the face structure encode relative orientations of the topological components. Note that the rotation system of $G$ is not sufficient to determine the orientation of topological components that are simple cycles.

Topological graphs $G$ and $H$ with more than one topological component are {\em isomorphic\/} if there is a one-to-one mapping between the vertices and edges of $G$ and $H$ satisfying the conditions (1)--(3) and, in addition, (4) the face structures of $G$ and $H$ are the same.


\subsection{Reduction to topologically connected graphs}\label{sub_5_2_topologically_connected}

Let $G$ be a graph with no isolated vertices.
Let $\mathcal{G}$ be a topological graph realizing $G$. If $\mathcal{G}$ has more than one topological component, we want to extend it to a topologically connected graph by adding edges connecting the topological components, in the same way as in the previous section. However, for this extension to be possible we may need to rearrange the topological components of $\mathcal{G}$, which changes the face structure of $\mathcal{G}$. While preserving the isomorphism classes of the $k$ topological components of $\mathcal{G}$, there are $2^k$ ways of choosing their orientation and at most $O(n^4)^{2k}$ possible face structures of topological graphs built from these components. Thus there are at most $2^{O(n\log n)}$ rearrangements of topological components of $\mathcal{G}$. Hence, by the same argument as in the previous section, we may further assume that $\mathcal{G}$ is topologically connected.


\subsection{Arrangements of pseudochords}
An essential part of the structure of a particular isomorphism class of simple topological graphs is captured by the following combinatorial object, which slightly generalizes arrangements of pseudolines.

An {\em arrangement of pseudochords\/} is a finite set $M$ of simple curves in the plane with endpoints on a common simple closed curve $C_M$, such that all the curves from $M$ lie in the region bounded by $C_M$ and every two curves in $M$ have at most one common point, which is a proper crossing. The elements of $M$ are called {\em pseudochords}. The arrangement $M$ is {\em simple\/} if no three pseudochords from $M$ share a common crossing. The {\em perimetric order\/} of $M$ is the counter-clockwise cyclic order of the endpoints of the pseudochords of $M$ on $C_M$.
The perimetric order of $M$ determines which pairs of pseudochords cross and which do not, but it does not determine the orders of crossings on the pseudochords. Two (labeled) arrangements of pseudochords are {\em isomorphic\/} if they have the same perimetric order and the same orders of crossings on the corresponding pseudochords. Equivalently, one arrangement can be obtained from the other one by an orientation preserving homeomorphism. Note that a $\mathcal{T}$-representation of a simple topological graph can be regarded as a simple arrangement of pseudochords.

The following proposition is inspired by Felsner's~\cite{F97_number} enumeration of simple wiring diagrams. Originally it appeared in~\cite{K09_enumeration} as Proposition 7, but in an incorrect, stronger form.

\begin{proposition}\label{prop_2na2k_chords}{\rm\cite[a correct form of Proposition 7]{K09_enumeration}}
The number of isomorphism classes of simple arrangements of $n$ pseudochords with fixed perimetric order inducing $k$ crossings is at most $2^{2k}$.
\end{proposition}

\begin{proof}
Let $M=\{p_1,p_2,\dots,p_n\}$ be a simple arrangement of pseudochords with endpoints on a circle $C_M$ and with a given perimetric order. Cut the circle at an arbitrary point and unfold it by a homeomorphism to a horizontal line $l$, while keeping all the pseudochords above $l$. Orient each pseudochord $p_i$ from its left endpoint $a_i$ to its right endpoint $b_i$.
Let $k_i$ be the number of crossings on $p_i$ and let $c^i_1,c^i_2,\dots,c^i_{k_i}$ be the crossings of $p_i$ ordered from $a_i$ to $b_i$. Let $p_{r(i,j)}$ be the pseudochord that crosses $p_i$ at $c^i_j$.

For two crossing pseudochords $p_i$ and $p_j$ we 
say that $p_i$ {\em is to the left of} $p_j$ if $a_i$ is to the left of $a_j$. This is equivalent with the rotation of their common crossing being $(a_i,b_j,a_j,b_i)$.

To each $p_i$ we assign a vector $\alpha^i=(\alpha^i_1, \alpha^i_2, \dots,\alpha^i_{k_i}) \in \{0,1\}^{k_i}$ where $\alpha^i_j=0$ if $p_{r(i,j)}$ is to the left of $p_i$ and $\alpha^i_j=1$ if $p_i$ is to the left of $p_{r(i,j)}$.

The sum of the lengths of the vectors $\alpha^i$ is equal to $\sum_{i=1}^n k_i=2k$. Hence, there are at most $2^{2k}$ different sequences $(\alpha^1,\alpha^2,\dots,\alpha^n)$ encoding an arrangement with the given perimetric order and the chosen orientation of pseudochords.

It remains to show that we can uniquely reconstruct the isomorphism class of $M$ from the vectors $\alpha^1,\alpha^2,\dots,\alpha^n$ by identifying the pseudochords $p_{r(i,j)}$.
We proceed by induction on $k$ and $n$. For arrangements without crossings there is only one isomorphism class with a fixed perimetric order. Now, suppose that we can reconstruct the isomorphism class for arrangements with at most $k-1$ crossings and take a sequence $\alpha=(\alpha^1,\alpha^2,\dots,\alpha^n)$ encoding an arrangement $M$ with $k$ crossings.

If some of the vectors $\alpha^i$ is empty, the corresponding pseudochord $p_i$ is empty (has no crossing). We may then draw $p_i$ as an arbitrary curve $\gamma_i$ from $a_i$ to $b_i$ in the upper half-plane of $l$. Then we split the arrangement into two parts: the {\em inner\/} part consisting of pseudochords with endpoints between $a_i$ and $b_i$, and the {\em outer\/} part with endpoints to the left of $a_i$ or to the right of $b_i$. We draw both parts separately by induction. Finally, by applying a suitable homeomorphism we place the inner part inside the region bounded by $\gamma_i$ and $l$ and the outer part outside that region.

Further we assume that $M$ has no empty pseudochords.

Without loss of generality we may assume that the left endpoints are ordered along $l$ as $a_1,a_2,\dots,a_n$ from left to right. Clearly, $\alpha^{1}=(1,1,\dots,1)$ and $\alpha^{n}=(0,0,\dots,0)$. It follows that there exists $s\in\{1,\dots,n-1\}$ such that $\alpha^{s}_1=1$ and $\alpha^{s+1}_1=0$. 

\begin{figure}
\begin{center}
\epsfbox{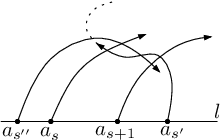}
\end{center}
\caption{$p_{s'}$ cannot be the first pseudochord crossing $p_s$.}
\label{obr_12_vynuceni}
\end{figure}

\begin{claim_}
The first crossing on the pseudochords $p_s$ and $p_{s+1}$ is their common crossing. That is, $r(s,1)=s+1$ and $r(s+1,1)=s$.
\end{claim_}

\begin{proof}[Proof of claim.]
Refer to Figure~\ref{obr_12_vynuceni}. For contradiction, suppose that $r(s,1)=s'\ge s+2$ (the case when $r(s+1,1)\le s-1$ is symmetric). Then $r(s+1,1)\notin\{s,s'\}$. Hence, $r(s+1,1)=s''$ for some $s''<s$ and the crossing of $p_{s+1}$ with $p_{s''}$ occurs within the triangle $a_sa_{s+1}c^s_1$. This forces the pseudochords $p_{s'}$ and $p_{s''}$ to cross twice, a contradiction.
\end{proof} 

Let $c=c^{s}_1=c^{s+1}_1$ be the first crossing on $p_s$ and $p_{s+1}$. Since the two arcs $a_sc$ and $a_{s+1}c$ are free of crossings, there is no endpoint between $a_s$ and $a_{s+1}$ on $l$. For the induction step, we swap the endpoints $a_{s}$ and $a_{s+1}$ in the perimetric order of $M$ and delete the first value from the vectors $\alpha^{s}$ and $\alpha^{s+1}$. In this way we obtain an encoding $\alpha'$ of an arrangement $M'$ with $k-1$ crossings, which is obtained from $M$ by deleting the arcs $a_sc$ and $a_{s+1}c$, including a small open neighborhood of $c$.
By the induction hypothesis, the isomorphism class of $M'$ can be uniquely reconstructed from $\alpha'$. By attaching to $M'$ two crossing arcs starting at $a_s$ and $a_{s+1}$ and thus extending the two pseudochords $p_s$ and $p_{s+1}$, we obtain an arrangement isomorphic to $M$. 
\end{proof}    


\subsection[Topologically connected topological graphs]{Counting isomorphism classes of topologically connected topological graphs}

Let $G$ be a graph with $n$ vertices, $m$ edges and no isolated vertices.
Let $\mathcal{G}$ be a topologically connected simple topological graph that realizes $G$.
The isomorphism class of $\mathcal{G}$ is determined by the isomorphism class of a $\mathcal{T}$-representation of $\mathcal{G}$. To determine the isomorphism class of a $\mathcal{T}$-representation, we need to determine (1) the topological spanning tree $\mathcal{T}$, (2) the perimetric order of the $\mathcal{T}$-representation, and (3) the isomorphism type of the induced arrangement of pseudochords.

(1) By Lemma~\ref{lemma_trees}, there are at most $2^{O(n\log n)}$ choices for the topological spanning tree $\mathcal{T}$ of $\mathcal{G}$, up to a homeomorphism of the plane. For the rest of the section, we fix one topological spanning tree $\mathcal{T}$ of $\mathcal{G}$.

(2) With $\mathcal{T}$ fixed, a $\mathcal{T}$-representation can have at most $2^{O(mn\log n)}$ diferent perimetric orders, as we have seen in Subsection~\ref{sub_pocitani_T_reprezentaci}.

This estimate is good enough when $G$ has $m= \omega(n\log n)$ edges, but we need a better upper bound for sparser graphs. This can be achieved by counting only perimetric orders that induce at most ${m \choose 2}$ crossings.

There are at most ${4mn \choose 8n}\le 2^{O(n\log n)}$ ways of choosing the set of endpoints of the pseudochords along the boundary of the disc $D$ in the $\mathcal{T}$-representation. To determine the perimetric order, we need, in addition, to determine a perfect matching of the endpoints inducing at most ${m \choose 2}$ crossings.

Such matchings can be also regarded as representations of {\em circle graphs\/} with a given number of vertices and edges. In the literature, these structures are called {\em chord diagrams}~\cite{K00_chord_diagrams_vsechny,R79_chord_intersection}. See Figure~\ref{obr_12_2_sawtooth}, left. Following the notation in~\cite{R79_chord_intersection}, let $C(n,k)$ denote the number of diagrams of $n$ chords with $k$ crossings. It is well known that $C(n,0)$, which is the number of noncrossing perfect matchings of $2n$ points on the circle, is equal to the $n$th Catalan number. Precise enumeration results for $C(n,k)$ in the form of generating functions were obtained by Touchard~\cite{T50_contribution} and Riordan~\cite{R75_distribution}, but explicit formulas for $C(n,k)$ were computed only for $k\le 6$~\cite{T50_contribution}.
The following asymptotic upper bound is implicit in Read's paper~\cite{R79_chord_intersection}.

\begin{proposition}{\rm\cite{R79_chord_intersection}}\label{prop_chord_diagrams}
For the number of diagrams of $n$ chords with at most $k$ crossings, we have the upper bound
$$\sum_{i=0}^k C(n,i) \le C(n){n+k \choose n}$$            
where $C(n)$ is the $n$th Catalan number.
\end{proposition}

\begin{proof} Like in the proof of Proposition~\ref{prop_2na2k_chords}, the key ``trick'' is breaking the symmetry of the circle by cutting it at one point and unfolding onto a horizontal line $l$. The chords then become arcs in the upper half-plane with endpoints on $l$. Each such arc has a distinguished left endpoint and a right endpoint. Instead of arbitrary arcs, Read~\cite{R79_chord_intersection} constructs triangular ``teeth'' consisting of a diagonal segment from the left endpoint followed by a vertical segment to the right endpoint and calls the resulting drawing the {\em sawtooth diagram\/} associated to the original chord diagram. See Figure~\ref{obr_12_2_sawtooth}, right.

\begin{figure}
\begin{center}
\epsfbox{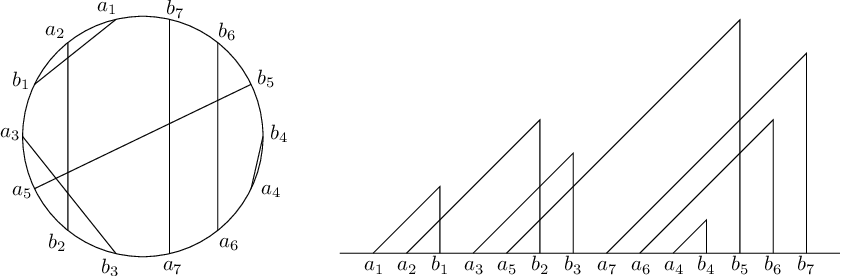} 
\end{center}
\caption{A chord diagram with seven chords and six crossings and a corresponding sawtooth diagram with $\kappa=(1,2,1,0,2,0,0)$.}
\label{obr_12_2_sawtooth}
\end{figure}

Let $L$ be the set of all the left endpoints of the chords on $l$, and $R$ the set of all the right endpoints. For every point $x$ on $l$, there are at least as many left endpoints than right endpoints to the left of $x$. Therefore the sets $L$ and $R$ correspond to the sets of $n$ left and $n$ right parentheses that are correctly matched. There are exactly $C(n)$ such partitions $(L,R)$ of the $2n$ points on $l$.

One partition $(L,R)$ can be shared by more sawtooth diagrams, if crossings are allowed. To determine the sawtooth diagram (and the corresponding chord diagram) uniquely, we encode the intersection graph of the chords as follows. 
Let $b_1,b_2,\dots,b_n$ be the points of $R$ ordered from left to right.
For $i=1,2,\dots,n$, let $c_i$ be the chord with right endpoint $b_i$, let $a_i$ be the left endpoint of $c_i$ and let $k_i$ be the number of chords with left endpoint to the right of $a_i$ that cross $c_i$. We claim that the vector $\kappa=(k_1, k_2, \dots, k_n)$, together with the partition $(L,R)$, uniquely determines the sawtooth diagram. This can be seen by drawing the diagram from left to right. All the crossings of the chord $c_i$ with chords with left endpoint to the right of $a_i$ occur on the vertical segment of $c_i$ with endpoint $b_i$. Therefore, every time we reach the $x$-coordinate of some $b_i$, we take the $(k_i+1)$th diagonal segment from the bottom and connect its right endpoint by a vertical line to $b_i$. All the other diagonal segments are extended further to the right.

Since $\sum_{i=1}^n k_i \le k$, for every partition $(L,R)$ there are at most ${n+k \choose k}$ possible vectors $\kappa$ and the proposition follows.
\end{proof}

By Proposition~\ref{prop_chord_diagrams}, by the entropy bound for binomial coefficients and by the inequality $\log_e(1+x)\le x$, the number of possible perimetric orders of the $\mathcal{T}$-represen\-tation is at most 
\begin{align*}
2^{O(n\log n)}\cdot C(2mn){2mn+{m \choose 2} \choose 2mn} &\le 2^{2mn \log(1+\frac{m}{4n})+\frac{m^2}{2}\log(1+\frac{4n}{m}) +4mn +O(n\log n)}\\
&\le 2^{2mn (\log(1+\frac{m}{4n})+2+\log_2 e) +O(n\log n)}.
\end{align*}

(3) By Proposition~\ref{prop_2na2k_chords}, there are less than $2^{m^2}$ isomorphism classes of simple arrangements of pseudochords induced by the $\mathcal{T}$-representation with a given perimetric order. Together with Proposition~\ref{prop_chord_diagrams} and previous discussion, this implies that
$$T(G) \le 2^{m^2+2mn(\log(1+\frac{m}{4n})+3.443) +O(n\log n)}.$$ 



For graphs with $m=O(n)$ the second term in the exponent becomes more significant. Since $m\ge n/2$, the exponent can be also bounded by 
$$m^2\cdot (1+8+4\log_2(9/8) + 1/2 \cdot \log_2 9)+o(m^2)\le 11.265 m^2+o(1),$$
using the entropy bound for the binomial coefficient ${4m^2+m^2/2 \choose 4m^2}$.
This proves the second upper bound in Theorem~\ref{veta_neizo}.


\subsubsection{Arrangements and quadrangulations}

Here we show an alternative approach to enumerating simple arrangements of pseudochords.

A {\em quadrangulation\/} of the disc $D$ is a $2$-connected plane graph embedded in $D$ such that its outer face coincides with the boundary of $D$ and every inner face is bounded by a $4$-cycle. A quadrangulation is called {\em simple\/} if it has no separating $4$-cycle. The vertices of the quadrangulation lying on the boundary of $D$ are called {\em external\/}, all the other vertices are {\em internal}.

Mullin and Schellenberg~\cite{MS68_c_nets_quadrangulations} proved that there are 
$$\frac{(3M+3)!(2N+M-1)!}{(M-1)!(2M+3)!N!(N+M+1)!} \le {3M+3\choose M}{2N+M-1\choose N}$$
isomorphism classes of rooted simple quadrangulations of the disc with $N$ internal and $2M+4$ external vertices. 

The {\em dual graph\/} of a simple arrangement of pseudochords is constructed as follows. Place one vertex inside each $2$-dimensional cell and one vertex in the interior of every boundary edge. Then join all pairs of vertices that correspond to adjacent $2$-cells or to a boundary edge and its adjacent $2$-cell. See Figure~\ref{obr_12_3_quadrangulation}.

\begin{figure}
\begin{center}
\epsfbox{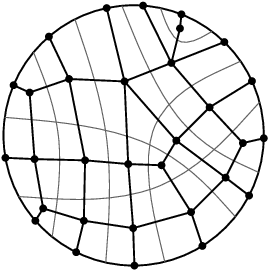} 
\end{center}
\caption{A simple arrangement of $7$ pseudochords with $9$ crossings and its dual quadrangulation.}
\label{obr_12_3_quadrangulation}
\end{figure}

Observe that the dual graph of a simple arrangement of $n$ pseudochords with $k$ crossings is a simple quadrangulation with $2n$ external and $n+k+1$ internal vertices. From the quadrangulation the original arrangement can be uniquely reconstructed up to isomorphism. However, not all simple quadrangulations can be obtained in this way: the graph of the $3$-dimensional cube is such an example.

By plugging $M=n-2$ and $N=n+i+1$ into Mullin's and Schellenberg's formula and summing over $i=0,1,\dots,k$ we obtain the following upper bound.

\begin{proposition}\label{prop_arrangmenty_pres_quadrangulace}
There are at most
$${3n-3\choose n-2}{3n+2k\choose n+k+1}$$
isomorphism classes of simple arrangements of $n$ pseudochords with at most $k$ crossings. 
\qed
\end{proposition}

Instead of using Proposition~\ref{prop_chord_diagrams} and~\ref{prop_2na2k_chords}, we may directly apply Proposition~\ref{prop_arrangmenty_pres_quadrangulace} with $n:=2mn$ and $k:={m \choose 2}$. This gives the first upper bound in Theorem~\ref{veta_neizo}:

\begin{align*}
T(G) &\le {6mn\choose 2mn}{m^2+6mn\choose \frac{m^2}{2}+2mn}\cdot 2^{O(n\log n)} \\  
     &\le 2^{m^2+2mn(1+3\log_2 3)+O(n\log n)} \\  
     &\le 2^{m^2+11.51mn+O(n\log n)}.
\end{align*}

Substituting $n\le 2m$, the exponent can be also bounded by 
$$m^2\cdot\left(4\log_2 3 + 8\log_2\frac{3}{2}+\frac{17}{2}\cdot\log_2 \frac{26}{17}+\frac{9}{2}\cdot \log_2\frac{26}{9}\right)+o(m^2) \le 23.118 m^2+o(1),$$ 
using the entropy bound for the binomial coefficients ${12m^2 \choose 4m^2}$ and ${13m^2 \choose 9m^2/2}$. 


\subsection{Upper bounds for very sparse graphs}\label{sub_very_sparse}

The upper bound $T(G) \le 2^{O(m^2)}$ is trivially obtained from the upper bound on the number of unlabeled plane graphs (or planar maps). Indeed, every drawing $\mathcal{G}$ of $G$ can be transformed into a plane graph $H$ by subdividing the edges of $\mathcal{G}$ by its crossings and regarding the crossings of $\mathcal{G}$ as new $4$-valent vertices in $H$. The graph $H$ has thus at most $n+ {m \choose 2}$ vertices, at most $m+2{m \choose 2}=m^2$ edges, no loops and no multiple edges. 

A {\em rooted connected planar map} is an unlabeled connected plane multigraph with a distinguished vertex, the {\em root}. In particular, multiple edges and loops are allowed.
Tutte~\cite{T63_census} showed that there are 
$$\frac{2(2M)!3^M}{M!(M+2)!}=2^{(\log_2(12)+o(1))M}$$ 
rooted connected planar maps with $M$ edges (see also~\cite{BR86_survey_asymptotic,BW85_loopless_planar,DM11_universal_exponents}). Walsh and Lehman~\cite{WL75_loopless} showed that the number of rooted connected planar loopless maps with $M$ edges is 
$$\frac{6(4M+1)!}{M!(3M+3)!}=2^{(\log_2(256/27)+o(1))M}.$$ 
This implies the upper bound $T(G)\le 2^{(\log_2(256/27)+o(1))m^2}$.
Somewhat better estimates could be obtained by reducing the problem to counting $4$-regular connected planar maps~\cite{RL01_4regular,RLL02_4regular_asymptoticky}, since typically almost all vertices in $H$ are the $4$-valent vertices obtained from the crossings of $\mathcal{G}$. But such a reduction would be less straightforward and the resulting upper bound $2^{(\frac{1}{2}\log_2(196/27)+o(1)) m^2}$ would not improve our upper bound $2^{m^2+O(mn)}$ for dense graphs (for graphs with more than $27n$ edges the first upper bound from Theorem~\ref{veta_neizo} is better). 

Note that by the reduction to counting planar maps, for every fixed constant $k$, we also obtain the upper bound $2^{O(km^2)}$ on the number of isomorphism classes of connected topological graphs with $m$ edges where all pairs of edges are allowed to cross $k$ times.


\section{The lower bounds}

In this section we present constructions of many pairwise different simple drawings of a given graph $G$, proving the lower bounds in Theorem~\ref{veta_neizo} and~\ref{veta_hlavni}.
Since we are dealing with arbitrary graphs, we use the following tool to find large subgraphs with more ``regular'' structure.

Let $G$ be a graph and let $A,B$ be disjoint subsets of its vertices. By $G[A,B]$ we denote the bipartite graph $(A\cup B, E_G(A,B))$ consisting of all edges with one endpoint in $A$ and the other endpoint in $B$.

\begin{lemma}\label{lemma_partitions}
Let $q,r$ be positive integers with $q\ge 3$ and $1\le r\le {q \choose 2}$. Let $H$ be a graph with vertex set $\{1,2,\dots,q\}$ and with $r$ edges. Let $G=(V,E)$ be a graph with $n$ vertices and $m$ edges. 
There is a partition of the vertex set $V$ into $q$ clusters $V_1, \dots, V_q$ such that for every edge $\{i,j\}$ of $H$ 
the number of edges in the bipartite graph $G[V_i,V_j]$ is at least
$$\frac{2m}{q^2}\left(1-  \sqrt{\frac{r(q-2)}{2}\cdot \frac{n}{m}} - O\left(\sqrt{\frac{m}{n^3}}\right)\right).$$

\end{lemma}

This is a variant of the result by K\"uhn and Osthus~\cite[Theorem 3]{KO07_several_cuts}, who consider the case of $r={q \choose 2}$ and assume that $G$ has maximum degree bounded by a constant fraction of $n$. 
The proof of Lemma~\ref{lemma_partitions} is similar to that of Theorem 3 in~\cite{KO07_several_cuts}. The main idea is to use the second order method to analyze the random partition. 

During the analysis we need to bound the number of pairs of adjacent edges in a graph $G$, which we denote by $p(G)$. Let $\mathcal{G}(n,m)$ be the class of all graphs with $n$ vertices and $m$ edges and let $f(n,m)$ be the maximum of $p(G)$ over all $G\in\mathcal{G}(n,m)$. Ahlswede and Katona~\cite{AK78_quasi} proved that the maximum of $p(G)$ is always attained for at least one of two special graphs in $\mathcal{G}(n,m)$, a {\em quasi-star\/} or a {\em quasi-clique\/}. \'Abrego et al.~\cite{AFNW09_sum_of_squares} completely characterized all graphs $G\in\mathcal{G}(n,m)$ for which $p(G)=f(n,m)$. 
The problem of computing $f(n,m)$ has been studied and partially solved by many researches; see~\cite{AFNW09_sum_of_squares} or~\cite{N07_sharp_asymptotics} for an overview of previous results.
Although all the values of $f(n,m)$ have been computed, the behavior of the function depends on certain nontrivial number-theoretic properties of the parameters $m,n$~\cite{AFNW09_sum_of_squares}.
Nikiforov~\cite{N07_sharp_asymptotics} proved tight asymptotic upper bounds on $f(n,m)$, which may be stated in a simplified form as follows.

\begin{lemma}\label{lemma_pocet_dvojic_sharp}
{\rm~\cite[Theorem 2]{N07_sharp_asymptotics}}
For all $n$ and $m$, 
\begin{align*} 
f(n,m)&\le \sqrt{2}m^{3/2}                       &\mbox{ if } m &\ge n^2/4, \mbox{ and }\\
f(n,m)&\le \frac{1}{2}\left((n^2-2m)^{3/2}-n^3\right) + 2nm &\mbox{ if } m &<n^2/4.
\end{align*}
\end{lemma}

We use a weaker, even more simplified asymptotic upper bound, which is easier to apply. For our purposes, we need the bound to be tight only for small values of $m$.

\begin{corollary}\label{cor_pocet_dvojic_approximation}
For all $n$ and $m$, 
$$f(n,m) \le \frac{1}{2}nm +O\left(\frac{m^2}{n}\right).$$
\end{corollary}

\begin{proof}
If $m \ge n^2/4$, then by Lemma~\ref{lemma_pocet_dvojic_sharp} we have 
$$f(n,m)\le\sqrt{2}m^{3/2}\cdot \frac{2m^{1/2}}{n} \le \frac{2\sqrt{2}m^2}{n}.$$ 
If $m < n^2/4$, then by Lemma~\ref{lemma_pocet_dvojic_sharp}, the desired upper bound is equivalent to the inequality
$$3mn-n^3+(n^2-2m)^{3/2} \le O(m^2/n).$$
Using the inequality $\sqrt{1-x}\le 1-x/2$, which holds for $x\le 1$, we have
$$3mn-n^3+(n^2-2m)^{3/2}=3mn+n^3((1-2m/n^2)^{3/2}-1) $$
$$\le 3mn + n^3\left(\left(1-\frac{m}{n^2}\right)^3-1\right) = \frac{3m^2}{n}-\frac{m^3}{n^3}.$$
\end{proof}

\begin{proof}[Proof of Lemma~\ref{lemma_partitions}]
Let $V_1,V_2,\dots,V_q$ be a random partition of the vertex set $V$, where each vertex is assigned independently to cluster $V_i$ with probability $1/q$. For $\{i,j\} \in E(H)$, let $X_{i,j}$ be a random variable counting the number of edges in the bipartite graph $G[V_i,V_j]$. Clearly, we have $\mathrm{E}X_{i,j}=2m/q^2$. Let $\sigma^2=\sigma_{i,j}^2=\mathrm{VAR}X_{i,j}$.

By Chebyshev's inequality, we have
$$P\left(X_{i,j}<\frac{2m}{q^2}-\sqrt{r}\sigma\right) < \frac{1}{r}.$$
It follows that there is a partition $V_1,V_2,\dots,V_q$ such that for every edge $\{i,j\}$ of $H$, the graph $G[V_i,V_j]$ has at least $\frac{2m}{q^2}-\sqrt{r}\sigma$ edges.

To complete the proof, we need to estimate $\sigma$ from above. Let $X=X_{i,j}$ for some $\{i,j\} \in E(H)$. We have

$$\sigma^2 = \mathrm{E}X^2 - (\mathrm{E}X)^2 = \mathrm{E}X^2 - \frac{4m^2}{q^4}.$$

For every edge $e$ of $G$, let $X_e$ be the indicator variable of the event that $e$ has one endpoint in $V_i$ and the other endpoint in $V_j$. Clearly, $X=\sum_{e\in E}X_e$. 
Recall that $p(G)$ denotes the number of pairs of adjacent edges in $G$. 
We have

\begin{align*}
\mathrm{E}X^2 &= \sum_{e\in E}\mathrm{E}X^2_e + 2\cdot\sum_{e,e'\in E; \, e\neq e'}\mathrm{E}X_eX_{e'} \\
&=\frac{2m}{q^2} + 2\cdot\frac{2}{q^3}\cdot p(G) + 2\cdot \frac{4}{q^4}\left({m\choose 2}-p(G)\right) \\
&= \frac{2m}{q^2} + \frac{8}{q^4}{m\choose 2} + \left(\frac{4}{q^3}-\frac{8}{q^4}\right)p(G). 
\end{align*}

By Corollary~\ref{cor_pocet_dvojic_approximation}, $p(G) \le \frac{1}{2}nm +O\left(\frac{m^2}{n}\right)$. Hence,

\begin{align*}
\sigma^2 &\le \frac{2m}{q^2} + \frac{4m^2}{q^4} + \frac{4q-8}{q^4}\cdot \frac{1}{2}nm +O\left(\frac{m^2}{n}\right) - \frac{4m^2}{q^4}\\
&\le \frac{2q-4}{q^4}\cdot nm + O\left(\frac{m^2}{n}\right) \\
&\le \left(\frac{\sqrt{2q-4}}{q^2}\cdot \sqrt{nm} +O\left(\frac{m^{3/2}}{n^{3/2}}\right) \right)^2
\end{align*}
and the lemma follows.

\end{proof}

\subsection{The lower bound in Theorem~\ref{veta_neizo}}

The construction giving the first lower bound in Theorem~\ref{veta_neizo} generalizes the construction from~\cite{K09_enumeration}.

Let $\varepsilon>0$ and let $G=(V,E)$ be a graph with $n$ vertices and $m$ edges. We apply Lemma~\ref{lemma_partitions} with $q=6$, $r=3$ and $E(H)=\{\{1,4\},\{2,5\},\{3,6\}\}$. If $m>(6+\varepsilon)\cdot n$, then Lemma~\ref{lemma_partitions} implies that there is a partition of $V$ into six clusters $V_1,V_2,\dots,V_6$ such that each of the three subgraphs $G[V_1,V_4],G[V_2,V_5],G[V_3,V_6]$ has $\Omega(m)$ edges. We may assume that $G[V_3,V_6]$ has the least number of edges of these three graphs.

Like in~\cite{K09_enumeration}, we construct $2^{\Omega(m^2)}$ drawings of $G$ that are all weakly isomorphic to the same geometric graph with vertices in convex position. For each $k=1,2,\dots, 6$, we place the vertices of the set $V_k$ on the unit circle, inside a small neighborhood of the point $(\cos(\frac{k\pi}{3}),\sin(\frac{k\pi}{3}))$; see Figure~\ref{obr_13_neizo_mrizka}, left. For every pair of vertices $u\in V_k$ and $v \in V_l$ such that $|k-l|\neq 3$, we draw the edge $uv$ as a straight-line segment. For $k\in \{1,2,3\}$, the edges between the sets $V_k$ and $V_{k+3}$ are drawn inside a narrow rectangle $R_k$ such that all the crossings among this group of edges occur outside the region $R=R_1 \cap R_2 \cap R_3$, and for $k,l \in \{1,2,3\}$, $k\neq l$, all the crossings between the edges of G$[V_k,V_{k+3}]$ and G$[V_l,V_{l+3}]$ lie inside $R$. In the region $R$, the edges connecting $V_2$ with $V_5$ form $\Omega(m)$ parallel curves. Together with the edges connecting $V_1$ with $V_4$, they form an $\Omega(m) \times \Omega(m)$ grid inside $R$. 

We partition the crossings of this grid into $\Omega(m)$ parallel diagonals forming horizontal rows. Each (horizontal) edge $e$ connecting $V_3$ with $V_6$ is drawn along one of the diagonal $d_i$. Each edge is assigned to a different diagonal. In the neighborhood of each crossing $c$ in $d_i$ we can decide whether the edge $e$ passes above or below $c$; see Figure~\ref{obr_13_neizo_mrizka}, right. These two possibilities give us two nonisomorphic topological graphs, and the choices can be made independently at each crossing of the grid. Since we make the choice at $\Omega(m^2)$ crossings, we obtain $2^{\Omega(m^2)}$ pairwise nonisomorphic drawings of $G$.

\begin{figure}
\begin{center}
\epsfbox{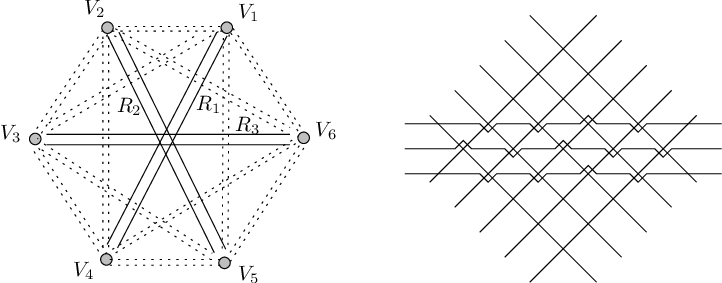}
\end{center}
\caption{A construction of $2^{\Omega(m^2)}$ pairwise nonisomorphic drawings of a given graph.}
\label{obr_13_neizo_mrizka}
\end{figure}

For graphs with superlinear number of edges, Lemma~\ref{lemma_partitions} gives a partition where each of the graphs $G[V_i,V_j]$ has $c=2m/q^2 - o(m)$ edges. In the previous construction, this gives a grid with $c^2=4m^2/q^4-o(m^2)$ crossings and hence $2^{3m^2/q^4-o(m^2)}$ pairwise nonisomorphic drawings of $G$, since $3/4$ of the crossings can be covered by $c$ parallel diagonals. For $q=6$, this gives the lower bound $T(G)\ge 2^{m^2/432-o(m^2)}$.

The constant $1/432$ can be easily improved. 
Previous construction used as a ``template'' a convex geometric drawing of $K_6$. This topological graph has one {\em free triangle\/}, that is, a triangular face bounded by three pairwise crossing edges. A free triangle may be {\em switched\/} by moving a portion of one of the boundary edges over the crossing of the other two edges. This feature is then amplified by replacing the free triangle by the grid construction. A set of $k$ free triangles is {\em independent\/} if no two of the triangles share a vertex. Equivalently, every two triangles share at most one boundary edge. This guarantees that each of the $2^k$ combinations of switched triangles is possible. There are simple drawings of $K_6$ with two independent free triangles~\cite{HM92_drawings,HM74_edges}. If we use one of them as a template, we get $2^{m^2/216 -o(m^2)}$ pairwise nonisomorphic drawings of $G$.

Using larger simple complete topological graphs as templates, much better lower bounds can be obtained. Instead of free triangles, we may consider, in general, {\em free $k$-tuples\/}, which consist of $k$ pairwise crossing edges with all ${k\choose 2}$ crossings close to each other, forming locally an arrangement of $k$ pseudolines. A system of free $k$-tuples is {\em independent\/} if no two $k$-tuples share a crossing.

When replacing a free $4$-tuple by the grid construction, we use both horizontal and vertical diagonals of the grid. After drawing the horizontal and vertical edges along the diagonals, half of the crossings in the grid become free $4$-tuples and the other half free triangles. Every free $4$-tuple can be drawn in $8$ different ways. Therefore, by replacing each of the original four edges by $c$ parallel edges, we obtain $2^{(1/2+3\cdot 1/2)c^2}=2^{2c^2}$ pairwise nonisomorphic drawings. That is, every free $4$-tuple in the template with $q$ vertices contributes $8m^2/q^4$ to the exponent in the lower bound on $T(G)$.

For example, the regular convex drawing of $K_{10}$ on Figure~\ref{obr_14_K10} has, after small perturbation, one free $5$-tuple, $5$ free $4$-tuples and $25$ free triangles, all independent.
Using this drawing as a template, we obtain the lower bound $T(G)\ge 2^{m^2\cdot 123 / 10^4 -o(m^2)} > 2^{m^2/82} -o(1)$ (for simplicity, we estimate the contribution of the free $5$-tuple by the contribution of a free $4$-tuple).

\begin{figure}
\begin{center}
\epsfbox{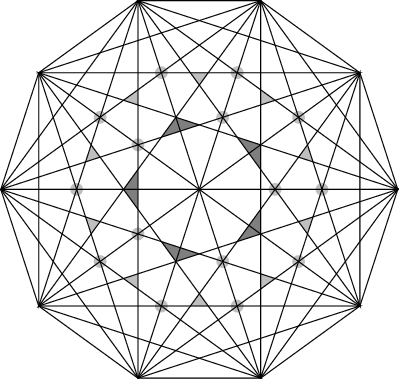}
\end{center}
\caption{A complete convex geometric graph with $10$ vertices and an independent system of $25$ free triangles (light grey), $5$ free $4$-tuples (dark grey) and one free $5$-tuple.}
\label{obr_14_K10}
\end{figure}

Further improvement can be obtained using all possible partial arrangements of three pairwise crossing systems of parallel pseudolines, in place of the grid construction, which produces only a subset of all such arrangements. Felsner and Valtr~\cite{FV11_pseudolines} proved that there are $2^{(4.5 \log_2 3 - 6-o(1))n^2}>2^{1.132\cdot n^2}-o(1)$ partial arrangements of $3n$ pseudolines that form three pairwise crossing subsets of $n$ parallel pseudolines. They observed that these partial arrangements are dual to rhombic tilings of a regular hexagon and used MacMahon's formula enumerating these tilings.
This also implies the rough lower bound $2^{2.264\cdot n^2}-o(1)$ on the number of partial arrangements of $4n$ pseudolines that form four pairwise crossing subsets of $n$ parallel pseudolines. Using these estimates with the template from Figure~\ref{obr_14_K10}, we obtain the lower bound $T(G)\ge 2^{m^2 \cdot 167.585/ 10^4  - o(m^2)} > 2^{m^2/60} - o(1)$.

This lower bound on $T(G)$ is very likely far from being optimal. However, it is probably hard to close the gap between the lower and upper bound on $T(G)$, given that even for pseudoline arrangements, the best known lower and upper bounds on their number differ significantly~\cite{FV11_pseudolines}.


\subsection{The lower bound in Theorem~\ref{veta_hlavni}}

Fix $\varepsilon>0$ and let $G$ be a graph with $n$ vertices and $m$ edges, with no isolated vertices, and satisfying at least one of the conditions $m>(1+\varepsilon)n$ or $\Delta(G)<(1-\varepsilon)n$. 

\subsubsection{The first construction}

First we show that $T_\mathrm{w}(G) \ge 2^{\Omega(m)}$ for graphs with $m>(4+\varepsilon) n$, generalizing a construction by Pach and T\'oth~\cite{PT04_how}.

Without loss of generality assume that $n$ is odd. Let $W$ be a random subset of $(n+1)/2$ vertices of $G$. The expected number of edges in the induced graph $G[W]$ is 
$$\frac{{(n+1)/2 \choose 2}}{{n \choose 2}}m=\frac{n+1}{4n}m = \left(\frac{1}{4}+\frac{1}{4n}\right)m.$$ 
Let $W_0$ be a subset of $(n+1)/2$ vertices inducing at least $(1/4+1/(4n))m$ edges. Every graph with $m$ edges has a bipartite subgraph with at least $m/2$ edges. Let $W_0=W_1\cup W_2$ be a bipartition such that the bipartite graph $G[W_1,W_2]$ has at least $(1/8+1/(8n))m$ edges.

We place the vertices of $V$ on three parallel vertical lines as follows.
The vertices of $W'=V\setminus W_0$ are placed on the $y$-axis to the points $(0,i/2)$, $i=0,1,\dots, (n-3)/2$, the vertices of $W_1$ to the points $(-1,i)$, $i=0,1,\dots, |W_1|-1$, and the vertices of $W_2$ to the points $(1,i), i=1,2,\dots, |W_2|-1$. Observe that the midpoint of every straight-line segment with one endpoint in $W_1$ and the other endpoint in $W_2$ lies in $W'$.

The idea of obtaining exponentially many pairwise different drawings of $G$ is now similar as in the grid construction in the previous subsection. The edges of $G[W']$ are drawn as arcs close to the $y$-axis. Every edge $e=w_1w_2$ of $G[W_1,W_2]$ is drawn as an arc along the straight-line segment $w_1w_2$, in one of two possible ways: either close above or close below the segment. See Figure~\ref{obr_15_dolni}. Let $w'$ be the midpoint of $w_1w_2$. If $w'$ is adjacent to a vertex  $u\in V\setminus\{w_1,w_2\}$, then in one of the two drawings the edges $w_1w_2$ and $w'u$ cross and in the other one they are disjoint. Since the minimum degree in $G$ is at least $1$, for every $w'\in W'$, there is at most one pair of vertices $w_1\in W_1$ and $w_2\in W_2$ such that $w'$ is a midpoint of the segment $w_1w_2$ and $w'$ is not adjacent to $V\setminus\{w_1,w_2\}$. This implies that for at least 
$$(1/8+1/(8n))m-(n-1)/2 > ((1/8+1/(8n)-1/(8+2\varepsilon))m = \Omega(m)$$ 
edges of $G[W_1,W_2]$, the two choices produce two weakly nonisomorphic drawings. Since the choices for all the edges are independent, this gives $2^{\Omega(m)}$ pairwise weakly nonisomorphic drawings of $G$ in total.

\begin{figure}
\begin{center}
\epsfbox{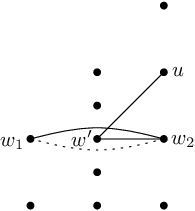}
\end{center}
\caption{Two ways of drawing the edge $w_1w_2$.}
\label{obr_15_dolni}
\end{figure}

\subsubsection{The second construction}
The lower bound on $T_\mathrm{w}(G)$ can be improved for sparse graphs with minimum degree at least $1$ that have at least $(1+\varepsilon)n$ edges or maximum degree at most $(1-\varepsilon)n$. Such assumptions are needed to guarantee a nontrivial number of pairs of independent edges, to avoid graphs like stars, which can only be drawn without crossings. For every such graph $G$ we show the lower bound $T_\mathrm{w}(G)\ge 2^{\Omega(n\log n)}$. Moreover, all the drawings in this construction are {\em geometric graphs\/}; that is, the edges are drawn as straight-line segments.

\begin{lemma}\label{lemma_les_bipartitni}
Let $G$ be a graph with $n$ vertices, $m$ edges, no isolated vertices and satisfying $m>(1+\varepsilon)n$ or $\Delta(G)<(1-\varepsilon)n$. Then the vertex set of $G$ can be partitioned into three parts $V_1, V_2,V_3$ such that $|V_1|\ge n/4$, every vertex from $V_1$ has a neighbor in $V_2$ and the induced graph $G[V_3]$ has at least $\lfloor\varepsilon/2 \cdot n\rfloor$ edges.
\end{lemma}

\begin{proof}
We distinguish two cases.

{\bf Case 1:} $G$ has a spanning forest $F$ with no isolated vertices such that its components can be partitioned into two subforests $F_1$ and $F_2$, each with at least $\lfloor\varepsilon n\rfloor$ vertices. Assume that $|V(F_1)|\ge |V(F_2)|$. We set $V_3=V(F_2)$. Now $V_1$ and $V_2$ are defined as the color classes of a proper $2$-coloring of $F_1$, with $|V_1|\ge |V_2|$.

{\bf Case 2:} No spanning forest as in Case 1 exists. Let $F$ be a spanning forest with no isolated vertices and maximum possible number of components. If some of the components has a path of length three as a subgraph, then by removing the middle edge of the path, the tree splits into two smaller nontrivial components, contradicting the choice of $F$. It follows that every component of $F$ is a star, that is, a graph isomorphic to $K_{1,k}$ for some $k\ge 1$.
Let $T_0$ be the largest component in $F$. By the assumption, $T_0$ is a star with more than $\lceil(1-\varepsilon)n\rceil$ vertices. In particular, $\Delta(G)\ge\Delta(T_0)\ge (1-\varepsilon)n$. Hence we have $m>(1+\varepsilon)n$. This means that $G$ has more than $\varepsilon n$ edges that do not belong to $T_0$. Let $V_3$ be the set of vertices spanned by $\lfloor\varepsilon/2 \cdot n\rfloor$ such edges, together with all vertices that do not belong to $T_0$. Finally, we set $V_2$ to be the one-element set containing the central vertex of $T_0$ and $V_1=V(T_0)\setminus(V_2\cup V_3)$.
\end{proof}

Let $V_1,V_2,V_3$ be the partition from Lemma~\ref{lemma_les_bipartitni}. Let $H$ be a bipartite subgraph of $G[V_3]$ with at least $\lfloor\varepsilon/4 \cdot n\rfloor$ edges. Split the set $V_3$ into two parts according to the bipartition of $H$ and place all vertices from one part in a small disc with center $(0,0)$ and radius $r<1/3$ and all vertices from the second part in a small disc with center $(1,0)$ and radius $r$, so that the vertices are in general position. Draw all edges of $G[V_3]$ as straight-line segments. See Figure~\ref{obr_16_dolni2}. There are two lines $t_1$, $t_2$ parallel to the $y$-axis going through points $(x_1,0)$ and $(x_2,0)$, respectively, such that $r<x_1<x_2<1-r$ and no two edges of $G[V_3]$ cross between $t_1$ and $t_2$. In particular, the edges of $G[V_3]$ split the vertical strip $S$ between $t_1$ and $t_2$ into at least $\lfloor\varepsilon/4 \cdot n\rfloor+1$ regions. Place all vertices of $V_2$ inside $S$ above the horizontal line with $y$-coordinate $r$, and each of the vertices of $V_1$ in one of the $\lfloor\varepsilon/4 \cdot n\rfloor+1$ regions of $S$, so that all vertices are in general position. Draw all the remaining edges as straight-line segments. The choice of the region for each vertex $v$ of $V_1$ determines how many edges from $H$ an edge connecting $v$ with $V_2$ crosses. In total, this gives $(|E(H)|+1)^{|V_1|}\ge (\varepsilon/4 \cdot n)^{n/4} \ge 2^{\Omega(n\log n)}$ pairwise weakly nonisomorphic geometric drawings of $G$.

\begin{figure}
\begin{center}
\epsfbox{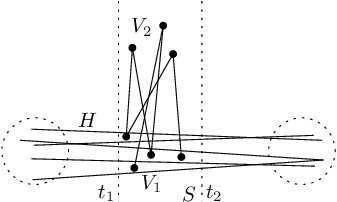}
\end{center}
\caption{An illustration of the second construction for the lower bound in Theorem~\ref{veta_hlavni}.}
\label{obr_16_dolni2}
\end{figure}


\section{Geometric graphs}\label{section_geometric}

A {\em geometric graph\/} is a topological graph where edges are drawn as straight-line segments. It is also usually assumed that the vertices are in general position, that is, no three of them lie on a line. 

For geometric graphs, asymptotically matching lower and upper bounds on both the number of  isomorphism and weak isomorphism classes can be easily derived from known results. It is easy to see that there are at least $2^{\Omega(n \log n)}$ weak isomorphism classes of complete geometric graphs with $n$ vertices, even when we drop the labels of vertices: place a set $A$ of $n/2$ points in convex position, draw the complete geometric graph on $A$ and distribute the remaining $n/2$ points in the $\Theta(n^4)$ bounded faces of $A$. The number of edges here is not crucial: the same asymptotic lower bound is known for matchings on $n$ vertices (see the remark after Theorem~\ref{veta_pseudosegmenty} in the Introduction).

For every given abstract graph $G$ with $n$ vertices, there are at most $2^{O(n \log n)}$ isomorphism classes of geometric graphs realizing $G$. This follows from the upper bound on the number of sign patterns~\cite{W68_nonlinear_manifolds}; see also~\cite[Theorem 6.2.1]{M02_lectures}. The reduction proceeds as follows. We define $2n$ variables as the $2n$ coordinates of the $n$ vertices. Every condition of the form ``segments $xy$ and $uv$ cross'', ``point $x$ is to the left of the ray $uv$'', ``the crossing of $xy$ with $uv$ is closer to $x$ than the crossing of $xy$ with $wz$", or ``vertices $x,y,z$ are seen from $u$ in clockwise order'', is then expressed in a straightforward way by inequalities of quadratic polynomials in the $2n$ variables. Then the theorem on the number of sign patterns is applied.

By the combinatorial definition of isomorphism in Subsection~\ref{sub_5_1_combinatorial_izomorphism}, this proves the upper bound for topologically connected geometric graphs. By the reduction in Subsection~\ref{sub_5_2_topologically_connected}, the upper bound holds also for general geometric graphs. 


\section{Concluding remarks and open problems} \label{section_last}

The problem of counting the asymptotic number of ``nonequivalent'' simple drawings of a graph in the plane has been answered only partially. Many open questions remain.

The gap between the lower and upper bounds on $T_{\mathrm{w}}(G)$ proved in Theorem~\ref{veta_hlavni} is wide open, especially for graphs with low density. For graphs with $cn^2$ edges, the lower and upper bounds on $\log T_{\mathrm{w}}(G)$ differ by a logarithmic factor. We conjecture that the correct answer is closer to the lower bound.

We do not even know whether $T_{\mathrm{w}}(G)$ is a monotone function with respect to the subgraph relation, since there are simple topological graphs that cannot be extended to simple complete topological graphs. See Figure~\ref{obr_8}, left, for an example.
Due to somewhat ``rigid'' properties of simple complete topological graphs, we have a much better upper bound for the complete graph than, say, for the complete bipartite graph on the same number of vertices.

\begin{problem}
Does the complete graph $K_n$ maximize the value $T_w(G)$ among the graphs $G$ with $n$ vertices? More generally, is it true that $T_w(H) \le T_w(G)$ if $H\subseteq G$?
\end{problem}

Our methods for proving upper bounds on the number of weak isomorphism classes of simple topological graphs do not generalize to the case of topological graphs with two crossings per pair of edges allowed. 

\begin{problem}
What is the number of weak isomorphism classes of drawings of a graph $G$ where every two independent edges are allowed to cross at most twice and every two adjacent edges at most once?
\end{problem}

For the complete graph with $n$ vertices, Pach and T\'oth~\cite{PT04_how} proved the lower bound $2^{\Omega(n^2\log n)}$ and the upper bound $2^{o(n^4)}$. 

A nontrivial lower bound can be proved also in the case when $G$ is a matching.
Ackerman et al.~\cite{APZ_touching_curves} constructed a system of $n$ $x$-monotone curves where every pair of curves intersect in at most one point where they either cross or touch, with $\Omega(n^{4/3})$ pairs of touching curves. Eyal Ackerman (personal communication) noted that this also follows from an earlier result by Pach and Sharir~\cite{PS91_vertical_visibility}, who constructed an arrangement of $n$ segments with $\Omega(n^{4/3})$ vertically visible pairs of disjoint segments. By changing the drawing in the neighborhood of every touching point, we obtain $2^{\Omega(n^{4/3})}$ different intersection graphs of $2$-intersecting curves, also called {\em string graphs of rank $2$}~\cite{PT04_how}. This improves the trivial lower bound observed by Pach and T\'oth~\cite{PT04_how}.

In section~\ref{section_uplne}, we proved that certain patterns are forbidden in the rotation systems of simple complete topological graphs, or more generally, in good abstract rotation systems. The problem of counting topological graphs was thus reduced to a combinatorial problem of counting permutations with forbidden patterns, by the recursion in Subsection~\ref{sub_dukaz_vety_uplne}. A general problem of this type can be formulated as follows. Given a constant $N$ and a collection $\mathcal{F}=\{F_1,F_2,\dots,F_m\}$ of sets of $N$-element permutation patterns, we say that a set $\mathcal{P}$ of permutations on $n$ elements is {\em $\mathcal{F}$-restricted\/} if for each $N$-tuple $X=(x_1,x_2,\dots,x_N)$ of positions, there is an $i \in \{1,2,\dots,m\}$ such that for every permutation $\pi\in \mathcal{P}$, all permutations from $F_i$ are forbidden as restrictions of $\pi$ at $X$. What is the maximum size of an $\mathcal{F}$-restricted set $\mathcal{P}$ of permutations on $n$ elements?

For example, in the special case of the Stanley-Wilf conjecture, the collection $\mathcal{F}$ consists of a single one-element set. A set of permutations with VC-dimension at most $k$ is an $\mathcal{F}$-restricted set where the collection $\mathcal{F}$ consists of $(k+1)!$ one-element sets, each containing a different permutation of $\{1,2,\dots,{k+1}\}$. 

In Subsection~\ref{sub_dukaz_vety_uplne}, we reduced the upper bound in Theorem~\ref{veta_uplne} to the upper bound on the size of an $\mathcal{F}$-restricted set where $\mathcal{F}$ consists of the following $2{N \choose 5}+2{N \choose 6}$ sets. For every set $A\subset\{1,2,\dots,N\}$ of five positions, the collection $\mathcal{F}$ contains a set $F_A$ of all permutations of $N$ elements whose restriction to $A$ is $(1,4,2,5,3)$ or some of its four cyclic shifts, and a set $F'_A$ of all permutations of $N$ elements whose restriction to $A$ is $(1,3,5,2,4)$ or some of its four cyclic shifts. Similarly, for every set $B\subset\{1,2,\dots,N\}$ of six positions, the collection $\mathcal{F}$ contains a set $F_B$ of all permutations of $N$ elements whose restriction to $B$ is $(1,2,3,6,5,4)$ or some of its five cyclic shifts, and a set $F'_B$ of all permutations of $N$ elements whose restriction to $B$ is $(1,4,5,6,3,2)$ or some of its five cyclic shifts. This follows from Lemma~\ref{lemma_C5}, \ref{lemma_T6} and from the proof of Theorem~\ref{theorem_Cm1_or_Tm2}, where the canonical linear ordering of the vertices of the unavoidable convex or twisted graphs is consistent with the linear ordering of the vertices of the given simple complete topological graph. 
Such $\mathcal{F}$-restricted sets of permutations are a special case of sets with VC-dimension smaller than $N$, which can have superexponential size~\cite{CK12_zobecnena_SW}, and generalize the sets with a single forbidden permutation pattern, for which a single exponential upper bound on their size is known~\cite{K00_FH_SW,MT04_SW}. Therefore one might ask for which collections $\mathcal{F}$ it is true that $\mathcal{F}$-restricted sets of permutations have only exponential size.

A positive answer to the following problem would improve the upper bound in Theorem~\ref{veta_uplne} to $T_{\mathrm{w}}(K_n)\le 2^{O(n^2)}$, which would be asymptotically optimal.

\begin{problem}\label{problem_permutacni}
Let $N>6$ be a constant positive integer. Let $\mathcal{P}$ be a set of permutations of $n$ elements such that for every $N$-tuple $X$ of positions, there is either a $5$-tuple $A\subset X$ such that the pattern $(1,3,5,2,4)$ and all its cyclic shifts are forbidden as restrictions at $A$, or a $5$-tuple $A'\subset X$ such that the pattern $(1,4,2,5,3)$ and all its cyclic shifts are forbidden as restrictions at $A'$, or a $6$-tuple $B\subset X$ such that $(1,2,3,6,5,4)$ and all its cyclic shifts are forbidden as restrictions at $B$, or a $6$-tuple $B'\subset X$ such that $(1,4,5,6,3,2)$ and all its cyclic shifts are forbidden as restrictions at $B'$. Is it true that $|\mathcal{P}|\le 2^{O(n)}$?
\end{problem}

For $N=4$ and $\mathcal{F}=\{\{(1,2,3,4)\},\{(3,4,1,2)\}\}$, a construction in~\cite{CK12_zobecnena_SW} shows an $\mathcal{F}$-restricted set of permutations of superexponential size. Such a construction does not necessarily satisfy the conditions in Problem~\ref{problem_permutacni} since, for example, the pattern $(3,4,1,2)$ is a restriction of just one cyclic shift of $(1,2,3,6,5,4)$, one cyclic shift of $(1,4,5,6,3,2)$ and of no cyclic shift of either $(1,3,5,2,4)$ or $(1,4,2,5,3)$. 
On the other hand, this construction does give a superexponential lower bound on the size of sets of permutations satisfying the restrictions implied by Lemma~\ref{lemma_abstract_C7} and Lemma~\ref{lemma_abstract_T816}, which appear in the proof of the combinatorial Theorem~\ref{veta_kombinatoricka}. This follows from the fact that every cyclic shift of the inverse of $(1,3,\dots,815,2,4,\dots,814)$ contains both patterns $(1,2,3,4)$ and $(3,4,1,2)$.
Therefore, a positive solution to Problem~\ref{problem_good} will require a different approach.

\section*{Acknowledgements}
The author thanks Josef Cibulka for discussions about enumerating various combinatorial objects.

\end{document}